\documentclass{amsart}
\usepackage[utf8]{inputenc}
\usepackage{amsmath}
\usepackage{amssymb}
\usepackage{amsthm}
\usepackage{enumitem}
\usepackage{todonotes}
\usepackage{hyperref}
\usepackage{url}
\usepackage{xcolor}

\setlist[enumerate]{label=\rm{(\arabic*)}, ref=(\arabic*)}

\DeclareMathOperator{\St}{St}
\DeclareMathOperator{\Aut}{Aut}
\DeclareMathOperator{\rist}{Rist}

\DeclareMathOperator{\Sym}{Sym}

\newcommand{\NN}{\mathbb{N}}
\newcommand{\ZZ}{\mathbb{Z}}
\newcommand{\MF}{\mathcal{MF}}
\newcommand{\B}{\mathcal{B}}

\newtheorem{thm}{Theorem}[section]
\newtheorem{lemma}[thm]{Lemma}
\newtheorem{prop}[thm]{Proposition}
\newtheorem{cor}[thm]{Corollary}

\theoremstyle{definition}

\newtheorem{defn}[thm]{Definition}
\newtheorem{rem}[thm]{Remark}
\newtheorem{notation}[thm]{Notation}

\newtheorem*{thm:MaxAreBranch}{Theorem \ref{thm:MaxSubgroupsOfBranchAreBranch}}
\newtheorem*{thm:Basilica}{Theorem \ref{thm:BasilicaIsInMF}}

\newcommand{\1}{\mathbf{1}}

\title[On maximal subgroups of infinite index in branch groups]{On maximal subgroups of infinite index in branch and weakly branch groups}
\author{Dominik Francoeur}

\begin{document}

\begin{abstract}
We generalise a technical tool, originally developed by Pervova for the study of maximal subgroups in Grigorchuk and GGS groups, to all weakly branch groups satisfying a natural condition, and in particular to all branch groups. We then use this tool to prove that every maximal subgroup of infinite index of a branch group is also a branch group. As a further application of this result, we show that every maximal subgroup of the Basilica group is of finite index.
\end{abstract}

\maketitle

\section{Introduction}

The study of maximal subgroups in branch groups dates back to the work of Pervova in 2000 \cite{Pervova00} and 2005 \cite{Pervova05}, where she showed that all the maximal subgroups of the periodic Grigorchuk and GGS groups are of finite index. Her original motivation was to shed some light on a conjecture of Kaplansky regarding the groups for which the Jacobson radical of a group algebra over a field of prime characteristic is equal to the augmentation ideal. However, her results also had other interesting implications for the groups she studied. For example, using the fact that every maximal subgroup of the Grigorchuk group is of finite index, Grigorchuk and Wilson showed in \cite{GrigorchukWilson03} that the Grigorchuk group is subgroup separable (also known as locally extended residually finite, or LERF), and that its finitely generated subgroups are either finite or abstractly commensurable with the entire group.

Pervova's results opened up a new line of enquiries in the field of branch groups. In \cite{BartholdiGrigorchukSunic03}, Bartholdi, Grigorchuk and \v{S}uni\'{k} asked if it was true that a maximal subgroup of a branch group must always be of finite index. This was answered negatively by Bondarenko in 2010 in \cite{Bondarenko10}, where he constructed an example of a branch group with maximal subgroups of infinite index. It was later shown by the author and Garrido in \cite{FrancoeurGarrido18} that even some Grigorchuk groups of intermediate growth admit maximal subgroups of infinite index (although they must necessarily contain elements of infinite order, by Pervova's result). In a different direction, Pervova's methods were extended to larger and larger classes of branch groups to prove that they admit no maximal subgroups of infinite index, notably by Alexoudas, Klopsch and Thillaisundaram \cite{AlexoudasKlopschThillaisundaram16}, and Klopsch and Thillaisundaram \cite{KlopschThillaisundaram18}.

At the heart of many of these results lies a technical lemma stating that for certain branch groups, the projections of vertex stabilisers of proper dense subgroups in the profinite topology are also proper and dense. This was proved by Pervova for the periodic Grigorchuk and GGS groups in her original paper, and was generalised by the author and Garrido in \cite{FrancoeurGarrido18} to self-replicating just-infinite branch groups with a regular and primitive action on the first level of a regular rooted tree. In this article, we prove that this result holds in fact for all weakly branch groups satisfying a natural condition, and in particular is true for all branch groups (Theorem \ref{thm:ProjectionsOFMaximalAreMaximal}). This gives us a very general tool to study maximal subgroups in branch and weakly branch groups. As an application, we use this result to show the following structural result about maximal subgroups of infinite index of branch groups.

\begin{thm:MaxAreBranch}
Every maximal subgroup of infinite index of a branch group is also a branch group.
\end{thm:MaxAreBranch}

As a further illustration of the usefulness of this tool, we then proceed to study the maximal subgroups of a weakly branch group known as the Basilica group. This group, which was introduced by Grigorchuk and \.{Z}uk in \cite{GrigorchukZuk02,GrigorchukZuk02b}, was the first example of an amenable but not subexponentially amenable group, as was shown by Bartholdi and Vir\'{a}g \cite{BartholdiVirag05}. We prove the following:
\begin{thm:Basilica}
Every maximal subgroup of the Basilica group is of finite index.
\end{thm:Basilica}
To the best of our knowledge, it is the first example of a group with this property among weakly branch but not branch groups.

The paper is organised as follows. In Section \ref{section:Preliminaries}, we recall a few definitions and key results regarding branch and weakly branch groups. In Section \ref{section:MaxSubgroups}, we prove our main result, Theorem \ref{thm:ProjectionsOFMaximalAreMaximal}, which states that projections of maximal subgroups of infinite index of branch groups are also maximal subgroups of infinite index. We then obtain Theorem \ref{thm:MaxSubgroupsOfBranchAreBranch} as a consequence. Finally, in Section \ref{section:Basilica}, we apply the results of Section \ref{section:MaxSubgroups} to the Basilica group to prove Theorem \ref{thm:BasilicaIsInMF}.

We would like to mention that a preliminary version of some of the results in Section \ref{section:MaxSubgroups}, and most of Section \ref{section:Basilica}, were contained in the author's PhD thesis \cite{FrancoeurThesis19}.

\subsection*{Acknowledgements} The author would like to thank Tatiana Nagnibeda, Anitha Thillaisundaram and Rostislav Grigorchuk for useful discussions and comments regarding this work.

\section{Preliminaries}\label{section:Preliminaries}

In this section, we define the basic notions and set the notation that will be used throughout the paper.

\subsection{Rooted trees}

Recall that a rooted tree is a tree (i.e. a simple graph without cycles) equipped with a distinguished vertex called the root. Throughout this entire text, unless otherwise specified, we will always assume that our rooted trees are infinite and locally finite.

Let $T$ be a rooted tree. By a slight abuse of notation, if $v$ is a vertex of $T$, we will write $v\in T$.

Since $T$ is a tree, for any pair of vertices $v,w\in T$, there exists a unique simple path between $v$ and $w$. The \emph{combinatorial distance} between $v$ and $w$, which we will denote by $d(v,w)$, is then the number of edges in this unique simple path.

\begin{defn}
Let $T$ be a rooted tree, let $v_0\in T$ be the root of $T$ and let $v\in T$ be an arbitrary vertex. The \emph{length} of $v$, denoted $|v|$, is defined as
\[|v|=d(v_0,v).\]
For $n\in \NN$, the set
\[L_n=\{v\in T \mid |v|=n\}\]
is called the \emph{n\textsuperscript{th} level} of $T$.
\end{defn}

The combinatorial distance also allows us to define a partial order on the vertices of $T$.

\begin{defn}
Let $T$ be a rooted tree of root $v_0\in T$. We define a partial order on the vertices of $T$ by saying that for $v,w\in T$, $v\leq w$ if and only if
\[d(v_0,v)+d(v,w)=d(v_0,w).\]
In other words, $v\leq w$ if and only if $v$ lies on the unique simple path connecting $w$ to the root.
\end{defn}

Using this partial order, we can define the notion of a subtree rooted at a specific vertex of a rooted tree.

\begin{defn}
Let $T$ be a rooted tree and $v\in T$ be an arbitrary vertex. We denote by $T_v$ the subtree of $T$ whose vertex set is the set of all vertices $w\in T$ such that $v\leq w$. By setting $v$ as the root, $T_v$ becomes a rooted tree that we call the \emph{subtree of $T$ rooted at $v$}.
\end{defn}

\subsection{Spherically homogeneous rooted trees}

In what follows, we will be interested in rooted trees which are as symmetric as possible. More precisely, we will concern ourselves with \emph{spherically homogeneous} trees.

\begin{defn}
Let $T$ be a rooted tree. Then, $T$ is said to be \emph{spherically homogeneous} if for all $n\in \NN$ and for all $v,w\in L_n$, we have $\text{deg}(v)=\text{deg}(w)$.
\end{defn}

To each spherically homogeneous rooted tree $T$, one can associate a sequence of integers $(m_i)_{i\in \NN}$, where $m_0$ is the degree of the root of $T$, and for $i>0$, $m_i+1$ is the degree of vertices on level $i$.

Conversely, given a sequence of integers $(m_i)_{i\in \NN}$, one can construct a spherically homogeneous rooted tree $T$ in the following way. For each $i\in \NN^*$, let $A_i$ be a set of cardinality $m_{i-1}$. Let $(A_i)^*$ be the set of all finite sequences of the form $a_1a_2\dots a_n$, where $a_i\in A_i$ for all $1\leq i \leq n$, including the empty sequence. We will call \emph{words in $(A_i)$} the elements of $(A_i)^*$. The \emph{length} of a word will simply be the number of elements in the sequence composing the word. We define a rooted graph $T$ by declaring its set of vertices to be $(A_i)^*$ and its root to be the empty word. Two words in $(A_i)$ are connected by an edge in $T$ if and only if one is a prefix of the other and their length differ by exactly $1$.

It is easy to check that the graph thus described is a spherically homogeneous rooted tree and that any other such tree with the same sequence of degrees must be isomorphic to it. Therefore, in what follows, when dealing with spherically homogeneous rooted trees, we will often implicitly assume that their vertex set is of the form $(A_i)^*$ for a given collection of finite sets $(A_i)$. Furthermore, by a slight abuse of notation, we will often simply write $T=(A_i)^*$ to mean the rooted tree whose vertex set is $(A_i)^*$.

In the special case where $(m_i)_{i\in \NN}$ is a constant sequence, the rooted tree $T$ is called a \emph{regular rooted tree}. In this case, we can choose a constant sequence for the sets $A_i$. We will thus write the vertex set of $T$ simply as $A^*$ for some finite set $A$ of cardinality $m_i$.

The advantage of representing vertices of a spherically homogeneous rooted tree as words in a sequence of finite sets is that it then gives us a canonical isomorphism between subtrees rooted at the same level. Indeed, let $T=(A_i)^*$ be a spherically homogeneous rooted tree and let $v\in (A_i)^*$ be a word of length $n$ for some $n\in \NN$. Then, it is clear from the definitions that the subtree $T_v$ is simply the subtree whose vertex sets is all the words in $(A_i)^*$ beginning with $v$. One can then define a bijection between the vertices of $T_v$ and the spherically homogeneous rooted tree $(A_{i+n})^*$ simply by deleting the prefix $v$. It is clear that this bijection is an isomorphism of rooted trees between $T_v$ and $(A_{i+n})^*$. In what follows, we will often implicitly identify these two rooted trees through this isomorphism. Since, for every $v\in L_n$, we have a canonical isomorphism between $T_v$ and the rooted tree $(A_{i+n})^*$, which does not depend on $v$, we also get a canonical isomorphism between $T_v$ and $T_w$ for all $v,w\in L_n$.

\subsection{Automorphisms of rooted trees}

Let $T$ be a rooted tree. We will denote by $\Aut(T)$ the group of all automorphisms of $T$. Since such an automorphism must preserve the root, it obviously cannot act transitively on the vertices of $T$. However, in the case of spherically homogeneous rooted trees, it acts \emph{spherically transitively}.

\begin{defn}
Let $T$ be a rooted tree and let $G\leq \Aut(T)$ be a group of automorphisms of $T$. We say that the action of $G$ on $T$ is \emph{spherically transitive} if $G$ acts transitively on the set $L_n$ for all $n\in \NN$.
\end{defn}

Let $T=(A_i)^*$ be a spherically homogeneous rooted tree and let $v=a_1a_2\dots a_n\in (A_i)^*$, $w=a_1a_2\dots a_n a_{n+1}\dots a_{n+k}\in (A_i)^*$ be two vertices of $T$, with $v\leq w$. As any automorphism of $T$ must preserve the partial order on its vertices, for any $g\in \Aut(T)$, we have $g\cdot v\leq g\cdot w$. In particular, this implies that there must exist a unique automorphism $g_{v} \in \Aut(T_{v})$ such that
\[g\cdot (va_{n+1}\dots a_{n+k}) = (g\cdot v)(g_v\cdot (a_{n+1}\dots a_{n+k})).\]

\begin{defn}
Let $T$ be a spherically homogeneous rooted tree. We will denote by $\varphi_v\colon \Aut(T) \rightarrow \Aut(T_v)$ the map that sends $g\in \Aut(T)$ to $g_v\in \Aut(T_v)$, and we will call this map the \emph{projection at $v$}.
\end{defn}

Note that the map $\varphi_v$ is not a homomorphism (unless $v$ is the root). However, it becomes one if we restrict it to $\St(v)$. Therefore, we will frequently want to restrict this projection map to elements that stabilise $v$. For convenience, let us introduce a notation for the image of a stabiliser under this map.

\begin{notation}\label{notation:Projections}
Let $T$ be a spherically homogeneous rooted tree and let $G\leq \Aut(T)$ be a group of automorphisms of $T$. We will denote by $G_v=\varphi_v(\St_G(v))$ the image of the stabiliser of $v$ in $G$ under the projection at $v$.
\end{notation}

Note that $G_v\leq \Aut(T_v)$. In the case where $T_v$ is naturally isomorphic to $T$, it could happen that $G_v$ is a subgroup of $G$. The groups for which this happens are called \emph{self-similar}.

\begin{defn}\label{defn:SelfSimilarGroup}
Let $T$ be a regular rooted tree and let $G\leq \Aut(T)$ be a group of automorphisms of $T$. Then, $G$ is said to be \emph{self-similar} if $G_v\leq G$ for all $v\in T$, and \emph{self-replicating} if $G_v=G$ for all $v\in T$. 
\end{defn}

Let $T=(A_i)^*$ be a spherically homogeneous rooted tree, and let $n\in \NN$. We will denote by $(A_i)_{i=1}^{n}$ the set of all words of length $n$ in $(A_i)^*$. Let $\pi_n\colon \Aut(T) \rightarrow \Sym((A_i)_{i=1}^n)$ be the natural map given by the action of $\Aut(T)$ on $(A_i)_{i=1}^n$, and let $\Aut((A_i)_{i=1}^n)$ be the image of $\Aut(T)$ under this map. It is clear that any automorphism $g\in \Aut(T)$ is uniquely determined by the collection $\{\pi_n(g), (g_v)_{v\in L_n}\}$. We thus have the following fact.

\begin{prop}\label{prop:Psin}
Using the notation above, the map
\begin{align*}
\psi_n\colon \Aut(T) &\rightarrow \Aut((A_i)_{i=1}^{n}) \ltimes \Aut((A_{i+n})^*)^{L_n}\\
g &\mapsto (\pi_n(g), (v\mapsto g_v))
\end{align*}
is an isomorphism, where $\Aut((A_{i+n})^*)^{L_n}$ denotes the set of maps from $L_n$ to $\Aut((A_{i+n})^*)$.
\end{prop}

In what follows, when considering a spherically homogeneous rooted tree $(A_i)^*$, we will often assume that we have an alphabetical order on each of the sets $A_i$, which then allows us to order the vertices of the $n$\textsuperscript{th} level of the tree by lexicographical order. Using this order, we can represent elements of $\Aut((A_{i+n})^*)^{L_n}$ simply as an $|L_n|$-tuple of elements of $\Aut((A_{i+n})^*)$. Thus, for $g\in \Aut(T)$, we will write
\[\psi_n(g) = \sigma (g_1, \dots, g_{|L_n|})\]
with $\sigma\in \Aut((A_i)_{i=1}^{n})$ and $g_j\in \Aut((A_{i+n})^*)$ for all $1\leq j \leq |L_n|$. 

\subsection{The boundary of a rooted tree}

Let $T$ be a rooted tree, and let $\partial T$ be the set of all infinite simple path in $T$ starting at the root. Given a path $\xi \in \partial T$ and a vertex $v\in T$, we will write $v\leq \xi$ if the path $\xi$ passes through $v$. Note that this relation is compatible with the partial order on the vertices of $T$, in the sense that if $u\leq v$ and $v\leq \xi$, then $u\leq \xi$, where $u,v\in T$ and $\xi\in \partial T$.

Given $v\in T$, the \emph{cylinder set at $v$} is the set
\[C_v=\{\xi \in \partial T \mid v\leq \xi\}.\]
One can show that the set of all cylinder sets forms a base for a topology on $\partial T$. The set $\partial T$ equipped with the topology defined by this base is called the \emph{boundary} of $T$. If no cylinder set contains only a single point (which is always the case, for instance, in a spherically homogeneous rooted tree $(A_i)^*$ with $|A_i|>1$), then the boundary $\partial T$ is homeomorphic to the Cantor set.

Any automorphism of a rooted tree $T$ can be naturally extended to a homeomorphism of the boundary $\partial T$. Furthermore, the support of such a homeomorphism must always be an open set, as the next lemma shows.

\begin{lemma}\label{lemma:SupportIsOpen}
Let $g\in \Aut(T)$ be an automorphism of the rooted tree $T$, and let
\[\text{Supp}(g)=\left\{\xi\in \partial T \mid g\xi\ne \xi\right\}\]
be the complement of the set of fixed points of the action of $G$ on the boundary of the tree $\partial T$ (we will call this set the support of $G$, even though we would need to take the closure of this set in order to obtain the classical notion of support). Then, $\text{Supp}(g)$ is an open set of $\partial T$.
\end{lemma}
\begin{proof}
Since $g$ is a homeomorphism of a first-countable Hausdorff space, its set of fixed points is closed, so $\text{Supp}(g)$ is open.

\end{proof}

\subsection{Stabilisers and rigid stabilisers}

Let $T$ be a rooted tree. Given a group $G\leq \Aut(T)$ of automorphisms of $T$, the action of $G$ on $T$ gives rise to many different subgroups. Notably, for $v\in T$, we have the \emph{vertex stabiliser} $\St_G(v)$, and for $n\in \NN$, we have the \emph{level stabiliser}
\[\St_G(n)=\bigcap_{v\in L_n}\St_G(v).\]
In addition to these stabilisers, we also have what is known as \emph{rigid stabilisers}.

\begin{defn}
Let $T$ be a rooted tree, let $G\leq \Aut(T)$ be a group of automorphisms of $T$, and let $U\subseteq T$ be a subset of vertices of $T$. The \emph{rigid stabiliser} $\rist_G(U)$ of $U$ is the subgroup of $G$ of all the elements that act trivially outside of the subtrees rooted at elements of $U$:
\[\rist_G(U)=\left\{g\in G \mid gv=v, \quad \forall v\notin \bigcup_{u\in U}T_u \right\}.\]
In the special case where $U=\{v\}$, we will simply write $\rist_G(v)$.
\end{defn}

In other words, the rigid stabiliser of a set $U$ is the subgroup of all elements of $G$ whose support is contained in $\bigcup_{u\in U}C_u\subseteq \partial T$.

Given an integer $n\in \NN$, we can also define the \emph{rigid stabiliser of the $n$\textsuperscript{th} level}.
\begin{defn}
Let $G\leq \Aut(T)$ be a group of automorphism of a rooted tree $T$, and let $n\in \NN$ be an integer. The \emph{rigid stabiliser of the $n$\textsuperscript{th} level}, $\rist_G(n)$, is the subgroup
\[\rist_G(n) = \left\langle \bigcup_{v\in L_n}\rist_G(v) \right\rangle\]
generated by the rigid stabilisers of each vertex in $L_n$.
\end{defn}

Given two vertices $u,v\in L_n$ on level $n$, elements in $\rist_G(u)$ and elements in $\rist_G(v)$ have support $C_u$ and $C_v$, respectively. Since those two sets are disjoint, we have
\[\rist_G(n)\cong \prod_{v\in L_n}\rist_G(v).\]

\subsection{Branch and weakly branch groups}

We are now ready to give the definition of a (weakly) branch group.

\begin{defn}
Let $T$ be a spherically homogeneous rooted tree, and let $G\leq \Aut(T)$ be a group of automorphism of $T$. If
\begin{enumerate}[label=(\roman*)]
\item $G$ acts spherically transitively,
\item $\rist_G(v) \ne \{1\}$ for every $v\in T$,
\end{enumerate}
then $G$ is called a \emph{weakly branch group}. If, furthermore,
\begin{enumerate}[label=(\roman*), resume]
\item $\rist_G(n)$ is of finite index in $G$ for all $n\in \NN$,
\end{enumerate}
then $G$ is said to be a \emph{branch group}.
\end{defn}

For self-similar groups (see Definition \ref{defn:SelfSimilarGroup}), we can also ask that the rigid stabilisers also satisfy some form of self-similarity, which gives rise to the notion of regular branch and weakly branch groups.

\begin{defn}\label{defn:RegularBranch}
Let $X$ be a finite alphabet and let $G\leq \Aut(X^*)$ be a self-similar weakly branch group of automorphisms of the rooted tree $X^*$. If there exists a non-trivial subgroup $K\leq G$ such that $K\leq \varphi_v(\rist_K(v))$ for all $v\in X$, the group $G$ is said to be \emph{regular weakly branch over $K$}. If, furthermore, $K$ is of finite index in $G$, we say that $G$ is \emph{regular branch over $K$}.
\end{defn}

The most direct example of a branch group is the group of automorphisms $\Aut(T)$ of a spherically homogeneous rooted tree $T$. Indeed, in that case, the action is spherically transitive and the rigid stabiliser is equal to the level stabiliser, and thus of finite index. Other important examples of branch groups include the Grigorchuk group \cite{Grigorchuk80,Grigorchuk83}, the Gupta-Sidki group \cite{GuptaSidki83} and their various generalisations (see for instance \cite{BartholdiGrigorchukSunic03}).

The following lemma gives us some information about the structure of normal and subnormal subgroups of branch and weakly branch groups. In this context, a version of it was first given by Grigorchuk in \cite{Grigorchuk00}, but it also appeared in various other contexts. Indeed, although we formulate it here in terms of weakly branch groups for simplicity, note that this result is true more generally for groups with micro-supported actions by homeomorphisms on topological spaces (see for instance \cite{Nekrashevych13}, Lemma 4.1). We include a proof here for completeness.

\begin{lemma}\label{lemma:SubnormalContainsDerivedRist}
Let $T$ be a spherically homogeneous rooted tree, let $G$ be a weakly branch group, and let $H\leq G$ be a $k$-subnormal subgroup of $G$ for some $k\in \NN$. If $v\in T$ is a vertex that is not stabilised by $H$, then $\rist_G^{(k)}(v)\leq H$, where $\rist_G^{(k)}(v)$ is the $k$\textsuperscript{th} derived subgroup of $\rist_G(v)$.
\end{lemma}
\begin{proof}
We proceed by induction on $k$. For $k=0$, we have $H=G$, so there is nothing to prove. Let us now assume that the result is true for some $k\in \NN$, and let us show it for $k+1$.

Let $v$ be a vertex not stabilised by $H$. Let $K\leq G$ be a $k$-subnormal subgroup of $G$ such that $H\trianglelefteq K$. Then, $v$ is not stabilised by $K$, so by our hypothesis, $\rist_G^{(k)}(v)\leq K$.

Let $h\in H$ be such that $hv\ne v$. For any $r_1, r_2\in \rist_G^{(k)}(v)$, we have $[h^{-1},r_1]\in H$ and $[[h^{-1},r_1],r_2]\in H$, since $H$ is normal in $K$ and $\rist_G^{(k)}\leq K$. On the other hand, we have $[h^{-1},r_1] = hr_1^{-1}h^{-1}r_1$, with $hr_1^{-1}h^{-1}\in \rist_G^{(k)}(hv)$. Since $hv\ne v$, this means that $hr_1^{-1}h^{-1}$ commutes with $r_1$ and $r_2$. Therefore,
\[[[h^{-1},r_1],r_2] = [r_1,r_2]\]
and so $[r_1,r_2]\in H$. This implies that $\rist_G^{(k+1)}(v)\leq H$ and thus concludes the proof.
\end{proof}

Of course, the previous lemma is only interesting if we know that $\rist_G^{(k)}(v)$ is not trivial, which we show in the following lemma. Once again, we include a proof for completeness.

\begin{lemma}\label{lemma:RistNotSolvable}
Let $G$ be a weakly branch group acting on a spherically homogeneous rooted tree $T$. Then, for any vertex $v\in T$ and for any $k\in \NN$, the subgroup $\rist_G^{(k)}(v)$ is non-trivial. In other words, $\rist_G(v)$ is not solvable.
\end{lemma}
\begin{proof}
We proceed by induction on $k$. The case $k=0$ follows directly from the definition of a weakly branch group.

Let us assume that there exists $k\in \NN$ such that $\rist_G^{(k)}(v)$ is non-trivial for all $v\in T$. Then,  for a given $v\in T$, there exist $g\in \rist_G^{(k)}(v)$ and $w\in T_v$ such that $gw\ne w$. By assumption, $\rist_G^{(k)}(w)$ is also non-trivial, and since $w\in T_v$, we have $\rist_G^{(k)}(w)\leq \rist_G^{(k)}(v)$. Let $r\in \rist_G^{(k)}(w)$ be a non-trivial element. We have
\[[g^{-1},r]= gr^{-1}g^{-1}r\in \rist_G^{(k)}(v).\]
We have $r\in \rist_G^{(k)}(w)$ and $gr^{-1}g^{-1}\in \rist_G^{(k)}(gw)$, and since $gw\ne w$, these two subgroups commute, which implies that $[g^{-1},r]\ne 1$. This concludes the proof.
\end{proof}

A consequence of the previous two lemmas is that commuting subnormal subgroups of weakly branch groups have disjoint support.

\begin{prop}\label{prop:CommutingSubgroupsDisjointSupport}
Let $G$ be a weakly branch group acting on a spherically homogeneous rooted tree $T$, and let $H_1, H_2\leq G$ be two subnormal subgroups of $G$. If $[H_1, H_2]=1$, then $\text{Supp}(H_1)\cap \text{Supp}(H_2)=\emptyset$, where
\[\text{Supp}(H_i) = \{\xi\in \partial T \mid h\xi \ne \xi \text{ for some } h\in H_i\}.\]
\end{prop}
\begin{proof}
Suppose on the contrary that $\text{Supp}(H_1)\cap \text{Supp}(H_2) \ne \emptyset$. This means that there exist $\xi \in \partial T$, $h_1\in H_1$ and $h_2\in H_2$ such that $h_i\xi \ne \xi$ for $i=1,2$. It follows that there exists $v\in T$ such that $h_i v\ne v$ for $i=1,2$. By Lemma \ref{lemma:SubnormalContainsDerivedRist}, there exists $k\in \NN$ such that $\rist_G^{(k)}(v)\leq H_i$ for $i=1,2$. This means that $\rist_G^{(k+1)}(v)\leq [H_1,H_2]$, which contradicts Lemma \ref{lemma:RistNotSolvable}.
\end{proof}

As a corollary, we get the following, which will be useful later on.

\begin{cor}\label{cor:CommutingSubnormalHaveVertexTrivialAction}
Let $G$ be a weakly branch group acting on a spherically homogeneous rooted tree $T$, and let $H_1, H_2\leq G$ be two subnormal subgroups of $G$ such that $[H_1, H_2]=1$. If $H_2$ is non-trivial, then there exists $v\in T$ such that $H_1\leq \St_G(v)$ and $\varphi_v(H_1)=1$.
\end{cor}
\begin{proof}
By Proposition \ref{prop:CommutingSubgroupsDisjointSupport}, $H_1$ and $H_2$ have disjoint support. Since $H_2$ is non-trivial, its support is non-empty. As it is open by Lemma \ref{lemma:SupportIsOpen}, there is a vertex $v\in T$ such that $C_v\subseteq \text{Supp}(H_2)$. This implies that $C_v\cap \text{Supp}(H_1)=\emptyset$, which means that $H_1$ acts trivially on $C_v$. Consequently, we must have $H_1v=v$ and $\varphi_v(H_1)=1$.
\end{proof}

\subsection{Maximal subgroups and prodense subgroups}

In what follows, we will be interested in the index of maximal subgroups. Let us fix some notation.

\begin{notation}
We will denote by $\MF$ the class of groups whose maximal subgroups are all of finite index.
\end{notation}

The existence of maximal subgroups of infinite index is closely related to the existence of proper prodense subgroups. Let us first recall the definition.

\begin{defn}
Let $G$ be a group. A subgroup $H\leq G$ is called \emph{prodense} if $HN=G$ for all non-trivial normal subgroup $N\trianglelefteq G$.
\end{defn}

Under a natural assumption, the existence of a maximal subgroup of infinite index is equivalent to the existence of a proper prodense subgroup.

\begin{prop}\label{prop:InfiniteIndexIffProdense}
Let $G$ be a finitely generated infinite group such that every proper quotient of $G$ belongs to the class $\MF$. Then, $G$ admits a proper prodense subgroup if and only if $G$ admits a maximal subgroup of infinite index.
\end{prop}
\begin{proof}
$(\Rightarrow)$ Let $H<G$ be a proper prodense subgroup of $G$. Since $G$ is finitely generated, $H$ is contained in a maximal subgroup $M<G$, which must also be prodense. In particular, $M$ cannot contain a non-trivial normal subgroup of $G$. This implies that $M$ must be of infinite index in $G$. Indeed, were it of finite index, it would contain a normal subgroup $N\trianglelefteq G$ of finite index, which would necessarily be non-trivial, since $G$ is infinite.

$(\Leftarrow)$ Let $M<G$ be a maximal subgroup of infinite index, and let $N\trianglelefteq G$ be a non-trivial normal subgroup. If $MN\ne G$, then $N\leq M$ by the maximality of $M$. Therefore, by the correspondence theorem, $M/N$ is a maximal subgroup of infinite index of $G/N$, which is absurd, since $G/N$ is in $\MF$. We conclude that $MN=G$ for all non-trivial normal subgroups of $G$, and so $M$ is a proper prodense subgroup of $G$.
\end{proof}

Note that if we are interested in determining whether a given finitely generated group $G$ is in $\MF$ or not, then the assumptions of Proposition \ref{prop:InfiniteIndexIffProdense} are fairly natural. Indeed, if the group is finite, then it is obviously in $\MF$, and if it admits a quotient which is not in $\MF$, then it cannot be in $\MF$ by the correspondence theorem. Therefore, the only unknown case is when $G$ is infinite and all its proper quotients are in $\MF$. As the next proposition shows, these two assumptions are always satisfied by branch groups.

\begin{prop}\label{prop:QuotientsOfBranchGroupsAreinMF}
Let $G$ be a finitely generated branch group. Then, $G$ is infinite and every proper quotient of $G$ is in $\MF$. In particular, $G$ admits a maximal subgroup of infinite index if and only if it admits a proper prodense subgroup.
\end{prop}
\begin{proof}
It is clear from the definition that branch groups are infinite. By Lemma \ref{lemma:SubnormalContainsDerivedRist}, if $N\trianglelefteq G$ is a non-trivial normal subgroup of $G$, there exists $n\in \NN$ such that $\rist_G^{(1)}(n)\leq N$. Since being in $\MF$ clearly passes to quotients, it suffices to show that $G/\rist_G^{(1)}(n)$ is in $\MF$ for all $n\in \NN$. Since $G$ is a branch group, $\rist_G(n)$ is of finite index in $G$, which implies that $G/\rist_G^{(1)}(n)$ is a finitely generated virtually abelian group. Such groups are known to be in $\MF$.
\end{proof}

For weakly branch groups, however, we do not know whether all quotients must be in $\MF$ or not.

\section{Maximal subgroups of infinite index in branch groups}\label{section:MaxSubgroups}

In this section, we prove that the projections of proper prodense subgroups of branch and weakly branch groups are again proper prodense subgroups (Theorem \ref{thm:ProjectionsOfProperDenseAreProper}). We then use this result to prove that maximal subgroups of infinite index of branch or weakly branch groups are also branch or weakly branch groups, respectively (Theorems \ref{thm:MaxSubgroupsOfBranchAreBranch} and \ref{thm:MaxSubgroupsOfWeaklyBranchAreWeaklyBranch}).

We begin with a lemma concerning the projections of prodense subgroups.

\begin{lemma}\label{lemma:ProjectionsOfDenseAreDense}
Let $T$ be a rooted tree, $G$ be a weakly branch group acting on $T$, $H\leq G$ be a prodense subgroup of $G$ and $u\in T$ be any vertex of the tree. Then, $H_u$ is a prodense subgroup of $G_u$.
\end{lemma}

\begin{proof}
It suffices to show that $H_u\rist_{G_u}^{(1)}(n) = G_u$ for all $n\in \NN$. Indeed, since $G_u$ is weakly branch, for every non-trivial normal subgroup $N\trianglelefteq G_u$, there exists $n\in \NN$ such that $\rist^{(1)}_{G_u}(n) \leq N$ by Lemma \ref{lemma:SubnormalContainsDerivedRist}. Furthermore, by Lemma \ref{lemma:RistNotSolvable}, $\rist^{(1)}_{G_u}(n)$ is non-trivial for every $n\in \NN$.

Let us suppose that $u$ is on level $m$, and let $n\in \NN$ be any natural number. Clearly, $\rist_G(m+n) \leq \St_G(m)\leq \St_G(u)$, and 
\[\varphi_u(\rist_G(n+m)) \leq \rist_{G_u}(n).\]
Therefore, $\varphi_u(\rist^{(1)}_G(n+m)) \leq \rist^{(1)}_{G_u}(n)$.

As $\rist^{(1)}_G(n+m)$ is a non-trivial normal subgroup of $G$, we have by hypothesis
\[H\rist^{(1)}_G(n+m) = G.\]
By definition, for every $g\in G_u$, there exists $\tilde{g} \in \St_G(u)$ such that $\varphi_u(\tilde{g}) = g$. Since $H\rist^{(1)}_G(n+m)=G$, there exist $h\in H$ and $r\in \rist^{(1)}_G(n+m)$ such that $hr = \tilde{g}$. Since $\tilde{g}, r\in \St_G(u)$, we must have $h\in \St_G(u)$. Therefore, we get
\[\varphi_u(h)\varphi_u(r) = g,\]
with $\varphi_u(h)\in H_u$ and $\varphi_u(r) \in \rist^{(1)}_{G_u}(n)$. This shows that $H_u\rist^{(1)}_{G_u}(n) = G_u$.
\end{proof}

As we have seen above, for weakly branch groups, the projection of any prodense subgroup to a vertex is still a prodense subgroup. However, to determine whether a group belongs to $\MF$ or not, we need to study \emph{proper} prodense subgroups. The next theorem tells us that the projections of proper prodense subgroups stay proper.




\begin{thm}\label{thm:ProjectionsOfProperDenseAreProper}
Let $G$ be a weakly branch group acting on a rooted tree $T$, $H\leq G$ be a prodense subgroup and $u\in T$ be any vertex. Then, $H$ is a proper subgroup of $G$ if and only if $H_u$ is a proper subgroup of $G_u$.
\end{thm}

\begin{proof}
If $H_u\ne G_u$, then clearly $H\ne G$. Let us now assume that $H\ne G$ and let us show that $H_u \ne G_u$. It suffices to prove this fact for $u\in L_1$, where $L_1$ is the first level of the rooted tree $T$. Indeed, if this property holds on the first level of the rooted tree, we can then use induction to prove it for $u$ on any level thanks to Lemma \ref{lemma:ProjectionsOfDenseAreDense}.

Therefore, let $u\in L_1$ be a vertex on the first level of the tree and let us assume for the sake of contradiction that $H\ne G$ but $H_u=G_u$.

The rigid stabiliser of the vertex $u$ in $H$, $\rist_H(u) = \rist_G(u) \cap H$, is a normal subgroup of $\St_H(u)$. Since $H_u=G_u$, it is also a normal subgroup of $\St_G(u)$. Indeed, for any $g\in \St_G(u)$, there exists $h\in \St_H(u)$ such that $\varphi_u(g) = \varphi_u(h)$. Hence, since any $r\in \rist_H(u)$ acts trivially outside of $T_u$, the subtree rooted at $u$, we have
\[grg^{-1} = hrh^{-1} \in \rist_H(u).\]
Since $\St_G(1) \leq \St_G(u)$, we have that $\rist_H(u) \trianglelefteq \St_G(1)$.

Now, since $G$ acts transitively on $L_1$ and since $H\St_G(1) = G$, we conclude that $H$ must also act transitively on $L_1$. Therefore, for any $v\in L_1$, there exists $h\in H$ such that $\St_H(v) = h\St_H(u)h^{-1}$. Hence,
\[H_v = \varphi_v(\St_H(v)) = h_u G_u h_u^{-1} = G_v\]
for all $v\in L_1$. It follows that $\rist_H(v) \trianglelefteq \St_G(1)$ for all $v\in X$. Therefore,
\[\rist_H(1) = \prod_{v\in X}\rist_H(v) \trianglelefteq \St_G(1).\]
Since $\rist_H(1) \trianglelefteq H$ and $H\St_G(1) = G$, we conclude that
\[\rist_H(1) \trianglelefteq G.\]
This implies that $\rist_H(1) = \{1\}$. Indeed, otherwise, by hypothesis, we would have $H\rist_H(1) = G$, which is absurd since $H\rist_H(1) = H$ and $H\ne G$.

Let $U\subset L_1$ be a subset of vertices of the first level such that $\rist_H(U)\cap \St_H(1) \ne \{1\}$ and whose cardinality is minimal for this property. Note that such a set must exist, since $\rist_H(L_1)\cap \St_H(1) = \St_H(1)$, which is clearly non-trivial. By transitivity, we can assume that $u\in U$. As we have seen above, $\rist_H(u)=\{1\}$, which implies that $2\leq |U| \leq |L_1|$.

We are going to show that there must exist a surjective homomorphism
\[\alpha\colon \varphi_u(\St_H(|v|)) \rightarrow \varphi_v(\St_H(|v|))\]
for some $v\in T \setminus T_u$.

Let us consider the (possibly trivial) subgroup $\rist_H(L_1\setminus \{u\}) \cap \St_H(1)$ of elements of $\St_H(1)$ that are sent to the identity by $\varphi_u$. We have
\begin{align*}
\left[\rist_H(L_1\setminus \{u\}) \cap \St_H(1), \rist_H(U)\cap \St_H(1)\right]&\leq \rist_H(L_1\setminus \{u\} \cap U) \cap \St_H(1) \\
&=\rist_H(U\setminus \{u\})\cap \St_H(1) \\
&=1,
\end{align*}
where the last equality is due to the minimality of $U$. We conclude that $\rist_H(L_1\setminus \{u\}) \cap \St_H(1)$ and $\rist_H(U)\cap \St_H(1)$ commute.

Now, let us fix some $w\in U\setminus \{u\}$. Notice that such a $w$ exists, since $|U|\geq 2$. Since $\varphi_w|_{\St_H(w)}$ is a homomorphism, we must have that $\varphi_w\left(\rist_H(L_1\setminus \{u\}) \cap \St_H(1)\right)$ and $\varphi_w\left(\rist_H(U)\cap \St_H(1)\right)$ commute. Notice that both of these subgroups are normal in $\varphi_w(\St_H(1))$, since $\rist_H(L_1\setminus \{u\}) \cap \St_H(1)$ and $\rist_H(U)\cap \St_H(1)$ are normal subgroups of $\St_H(1)$. Now, since $\St_H(1)$ is normal in $\St_H(w)$, and since $\varphi_w(\St_H(w)) = G_w$, we get that $\varphi_w(\St_H(1))$ is a normal subgroup of $G_w$. Thus, $\varphi_w\left(\rist_H(L_1\setminus \{u\}) \cap \St_H(1)\right)$ and $\varphi_w\left(\rist_H(U)\cap \St_H(1)\right)$ are both $2$-subnormal subgroups of $G_w$.

By the minimality of the size of $U$, we know that $\varphi_w\left(\rist_H(U)\cap \St_H(1)\right)$ is non-trivial. Therefore, it follows from Corollary \ref{cor:CommutingSubnormalHaveVertexTrivialAction} that there exists some $v\in T_w$ (possibly equal to $w$) such that $\rist_H(L_1\setminus \{u\}) \cap \St_H(1) \leq \St_H(v)$ and
\[\varphi_v\left(\rist_H(L_1\setminus \{u\}) \cap \St_H(1)\right)=1.\]

This implies that the kernel of the map
\[\varphi_u|_{\St_H(|v|)} \colon \St_H(|v|) \rightarrow \varphi_u(\St_H(|v|))\] is contained in the kernel of the map
\[\varphi_v|_{\St_H(|v|)} \colon \St_H(|v|) \rightarrow \varphi_v(\St_H(|v|)).\]
Therefore, we have a well-defined homomorphism
\begin{align*}
\alpha \colon \varphi_u(\St_H(|v|))&\rightarrow \varphi_v(\St_H(|v|))\\
g &\mapsto \varphi_v(\tilde{g})
\end{align*}
where $\tilde{g}\in \St_H(|v|)$ is any element such that $\varphi_u(\tilde{g}) = g$. Since, for any $\tilde{g}\in \St_H(|v|)$, we have $\alpha(\varphi_u(\tilde{g}))=\varphi_v(\tilde{g})$, we see that this homomorphism is surjective.

We will now derive a contradiction from the existence of this homomorphism and the fact that $H$ is prodense in $G$.

Let $r\in \rist_G(v)$ be any non-trivial element of the rigid stabiliser of $v$. Such an element exists since $G$ is a weakly branch group. Notice that since we have shown that $\rist_H(w)=1$ for all $w\in L_1$, we must have that $r\notin H$.

Let us consider $\varphi_v(r)\in \varphi_v(\St_G(|v|))$. Since $r$ is non-trivial and in the rigid stabiliser of $v$, we have that $\varphi_v(r)$ is non-trivial. Therefore, there must exist $n\in \NN$ such that $\varphi_v(r)\notin \St_{G_v}(n)$. Let us write
\[K=\alpha^{-1}(\St_{\varphi_v(\St_H(|v|))}(n)) \leq \varphi_u(\St_H(|v|)).\]
As $\St_{\varphi_v(\St_H(|v|))}(n)$ is a normal subgroup of finite index of $\varphi_v(\St_H(|v|))$, we have that $K$ is a normal subgroup of finite index of $\varphi_u(\St_H(|v|))$. Now, we know that $\St_H(|v|)$ is a normal subgroup of finite index in $\St_H(u)$. Therefore, we have that $\varphi_u(\St_H(|v|))$ is a normal subgroup of finite index of $\varphi_u(\St_H(u)) = G_u$. This means that $K$ is a subgroup of finite index of $G_u$. Let $L$ be the normal core of $K$ in $G_u$. Then, $L$ is a normal subgroup of $G_u$ of finite index. It follows from Lemma \ref{lemma:SubnormalContainsDerivedRist} and the fact that $G_u$ is a weakly branch group that there exists some $k\in \NN$ such that $\rist^{(1)}_{G_u}(k) \leq L$. Let us define $m$ as the maximum between $k$ and $|v|+n-1$.

Let us consider $\rist^{(1)}_G(m+1)$. We know that this is a normal subgroup of $G$. Thus, since $H$ is prodense, we have that $H\rist^{(1)}_G(m+1)=G$. In particular, this means there exist $h\in H$ and $g\in \rist^{(1)}_G(m+1)$ such that $hg=r$, where $r$ was defined above. Since $r\in \St_G(|v|)$ and since $\rist^{(1)}_G(m+1) \leq \St_G(m+1) \leq \St_G(|v|)$, we must have $h\in \St_G(|v|)$, and so $h\in \St_H(|v|)$.

Furthermore, since $\varphi_u(r)=1$, we find that $\varphi_u(h)=\varphi_u(g)^{-1}$. As $g\in \rist^{(1)}_G(m+1)$, we must have $\varphi_u(g)\in \rist^{(1)}_{G_u}(m) \leq L$. Therefore, $\varphi_u(h)\in L$.

Let us now consider $\varphi_v(r) = \varphi_v(h)\varphi_v(g)$. Since $h\in \St_H(|v|)$, we must have
\[\varphi_v(h) = \alpha(\varphi_u(h)) \in \alpha(L) \leq \alpha(K) \leq \St_{G_v}(n).\]
Since $g\in \rist_G(m+1)\leq \St_G(m+1) \leq \St_G(|v|+n)$, we must have $\varphi_v(g)\leq \St_{G_v}(n)$. Therefore, we get $\varphi_v(r)\in \St_{G_v}(n)$, which contradicts our choice of $r$ and $n$.

Since assuming that $H_u=G_u$ led us to a contradiction, we must conclude that $H_u\ne G_u$.

\end{proof}

We have just shown that the projections of proper prodense subgroups are proper prodense subgroups. We will now see that if we have a maximal subgroup of infinite index of a weakly branch group, then its projections are also maximal subgroups of infinite index.

Recall from Proposition \ref{prop:QuotientsOfBranchGroupsAreinMF} that for branch groups, every proper quotient is always in $\MF$.

\begin{thm}\label{thm:ProjectionsOFMaximalAreMaximal}
Let $G$ be a weakly branch group acting on a rooted tree $T$, and suppose that every proper quotient of $G$ is in $\MF$. If $M<G$ is a maximal subgroup of $G$ of infinite index, then $M_v$ is a maximal subgroup of infinite index of $G_v$ for every $v\in T$.
\end{thm}

\begin{proof}
Since $M$ is a maximal subgroup of infinite index and since every proper quotient of $G$ is in $\MF$, we have that $M$ is a proper prodense subgroup of $G$. Then, by Theorem \ref{thm:ProjectionsOfProperDenseAreProper}, we know that $M_v$ is a proper prodense subgroup of $G_v$. It thus only remains to show that it is maximal.

For the sake of contradiction, let us assume that $M_v$ is not maximal in $G_v$. Then, there exists $g\in G_v$ such that
\[M_v \lneq \langle M_v, g \rangle \lneq G_v.\]
By the definition of $G_v$, there exists $\tilde{g}\in \St_G(v)$ such that $\varphi_v(\tilde{g})=g$. Since $M$ is prodense in $G$, we have that $M\St_G(|v|)=G$. Therefore, there exists $\tilde{m}\in M$ and $\tilde{s}\in \St_G(|v|)$ such that $\tilde{g}=\tilde{m}\tilde{s}$. Since both $\tilde{g}$ and $\tilde{s}$ belong to $\St_G(v)$, we must have that $\tilde{m}\in \St_G(v)$. Therefore, we have
\[\varphi_v(\tilde{s}) = \varphi_v(\tilde{m}^{-1}\tilde{g}) = \varphi_v(\tilde{m}^{-1})g.\]
Since $\tilde{m}\in M$, we have $\varphi_v(\tilde{m}^{-1})\in M_v$. We conclude that $\langle M_v, g \rangle = \langle M_v,\varphi_v(\tilde{m}^{-1})g \rangle $. Thus, replacing $g$ by $\varphi_v(\tilde{m}^{-1})g$ and $\tilde{g}$ by $\tilde{s}$ if necessary, we can assume that $\tilde{g}\in \St_G(|v|)$.

Now, let $w\in L_{|v|}$ be any vertex of level $|v|$ different from $v$. Since $\rist_G(w)$ is a normal subgroup of $\St_G(w)$, we have that $\varphi_w(\rist_G(w))$ is a normal subgroup of $G_w$. By Theorem \ref{thm:ProjectionsOfProperDenseAreProper}, we know that $M_w$ is a proper prodense subgroup of $G_w$. Therefore, we have that $M_w\varphi_w(\rist_G(w)) = G_w$. Consequently, there exist $m_w \in M_w$ and $r_w\in \rist_G(w)$ such that $m_w\varphi_w(r_w) = \varphi_w(\tilde{g})$.

Let
\[\hat{g} = \tilde{g}\prod_{w\in L_{|v|}\setminus \{v\}}r_w^{-1} \in \St_G(|v|).\]
Then, for every $w\in L_{|v|}\setminus \{v\}$, we have
\begin{align*}
\varphi_w(\hat{g}) &= \varphi_w(\tilde{g})\prod_{w'\in X\setminus \{v\}}\varphi_{w'}(r_{w'})^{-1} \\
&= \varphi_w(\tilde{g})\varphi_w(r_w)^{-1} \\
&= m_w \in M_w
\end{align*}
where the second equality comes from the fact that $r_{w'}\in \rist_G(w')$, so $\varphi_w(r_{w'})=1$ if $w\ne w'$. Furthermore, by a similar computation, we have that $\varphi_v(\hat{g}) = g$.

Since $g\notin M_v$ by construction, we must have that $\hat{g}\notin M$. Let us write $H=\langle M, \hat{g} \rangle$. By the maximality of $M$, we must have that $H = G$. However, we will now prove that
\[H_v = \langle M_v, g \rangle \lneq G_v, \]
which will contradict the fact that $H = G$.

Let $h\in \St_H(v)$ be an arbitrary element of $H$ stabilising $v$. Since $h\in H = \langle M, \hat{g} \rangle$, there exist $n\in \NN$, $i_1,\dots, i_n \in \ZZ$ and $\mu_1, \dots, \mu_{n+1} \in M$ such that 
\[h = \mu_1 \hat{g}^{i_1} \mu_1^{-1} \mu_2 \hat{g}^{i_2}\mu_2^{-1} \dots \mu_n \hat{g}^{i_n} \mu_{n}^{-1} \mu_{n+1}.\]
Notice that since $\hat{g}\in \St_G(|v|)$, we must have that $\mu_{j}g^{i_j}\mu_{j}^{-1} \in \St_G(|v|)$ for all $1\leq j \leq n$, and since $h\in \St_G(v)$, this implies that $\mu_{n+1}\in \St_M(v)$.

We will now see that for all $1\leq j \leq n$, we must have
\[\varphi_v(\mu_j\hat{g}^{i_j} \mu_j^{-1}) \in \langle M_v, g \rangle.\]
Indeed, if $\mu_j\in \St_G(v)$, then we have
\[\varphi_v(\mu_j\hat{g}^{i_j} \mu_j^{-1}) = \varphi_v(\mu_j)g^{i_j}\varphi_v(\mu_j)^{-1} \in \langle M_v, g \rangle.\]
If $\mu_j\notin \St_G(v)$, then we have $\mu_j^{-1}\cdot v \ne v$. Let us set $w=\mu_j^{-1}\cdot v$. We have that 
\[\varphi_v(\mu_j\hat{g}^{i_j} \mu_j^{-1}) = \varphi_{w}(\mu_j)m_w^{i_j}\varphi_w(\mu_j)^{-1}.\]
Now, since $m_w^{i_j}\in M_w$, there exists some $\nu \in \St_M(w)$ such that $\varphi_w(\nu) = m_w^{i_j}$. It follows that
\[\varphi_v(\mu_j\hat{g}^{i_j} \mu_j^{-1}) =\varphi_v(\mu_j\nu \mu_j^{-1}).\]
Since $\nu\in \St_M(w)$ and since $\mu_j\in M$, we have that $\mu_j\nu \mu_j^{-1} \in \St_M(\mu_j\cdot w) = \St_M(v)$. Therefore, we conclude that $\varphi_v(\mu_j\hat{g}^{i_j} \mu_j^{-1})\in M_v$.

Finally, since $\mu_{n+1}\in \St_M(v)$, we have that $\varphi_v(\mu_{n+1})\in M_v$. Thus, we see that $\varphi_v(h) \in \langle M_v, g \rangle$. This concludes the proof.
\end{proof}

Using the above results, we can show that every maximal subgroup of infinite index of a branch group is again a branch group. For this, we first need a lemma.

\begin{lemma}\label{lemma:RistMax}
Let $G$ be a weakly branch group acting on a rooted tree $T$ such that every proper quotient of $G$ is in $\MF$, and let $M<G$ be a maximal subgroup of infinite index of $G$. Then, for all $n\in \NN$, we have $\rist_M(n)=\rist_G(n)\cap M$.
\end{lemma}
\begin{proof}
It is clear that $\rist_M(n)\leq \rist_G(n)\cap M$. Thus, we only need to show the other inclusion. Let $g\in \rist_G(n)\cap M$ be any element. Then, by the definition of $\rist_G(n)$, there exists for all $v\in L_n$ an element $g_v\in \rist_G(v)$ such that
\[g=\prod_{v\in L_n}g_v.\]
Since $g\in M$, we have $\varphi_v(g)\in M_v$ for all $v\in L_n$. Notice that we have $\varphi_v(g)=\varphi_v(g_v)$, since $\varphi_v(g_w)=1$ for all $w\ne v$.

Let us fix some $v\in L_n$, and let us define $H=\langle M, g_v \rangle$. Since $\varphi_v(g_v)\in M_v$ and $\varphi_w(g_v)=1$, it is easy to see that we have $H_v=M_v$. As $M_v$ is a proper subgroup of $G_v$ by Theorem \ref{thm:ProjectionsOfProperDenseAreProper}, we conclude that $H\ne G$. Therefore, by the maximality of $M$, we have $H=M$, which implies that $g_v\in \rist_M(v)$. As this is true for all $v\in L_n$, we get that $g\in \rist_M(n)$.
\end{proof}

\begin{thm}\label{thm:MaxSubgroupsOfBranchAreBranch}
Let $G$ be a branch group acting on a rooted tree $T$, and let $M<G$ be a maximal subgroup of infinite index of $G$. Then, $M$ is a branch group for its action on $T$.
\end{thm}
\begin{proof}
By Lemma \ref{lemma:RistMax}, for all $n\in \NN$, we have $\rist_M(n)=\rist_G(n)\cap M$. Therefore,
\[[M:\rist_M(n)] = [M:\rist_G(n)\cap M] \leq [G:\rist_G(n)]<\infty,\]
so $M$ is a branch group.
\end{proof}

We also have a corresponding result in the case where $G$ is a weakly branch group.

\begin{thm}\label{thm:MaxSubgroupsOfWeaklyBranchAreWeaklyBranch}
Let $G$ be a weakly branch group acting on a rooted tree $T$, and suppose that every proper quotient of $G$ is in $\MF$. Let $M<G$ be a maximal subgroup of infinite index of $G$. Then, $M$ is a weakly branch group for its action on $T$.
\end{thm}
\begin{proof}
By Lemma \ref{lemma:RistMax}, we know that $\rist_M(n)=\rist_G(n)\cap M$ for all $n\in \NN$. We thus need to show that $\rist_G(n)\cap M$ is non-trivial for every $n\in \NN$.

Let us fix some $n\in \NN$, and let $r_1\in \rist_G(n)$ be a non-trivial element of $\rist_G(n)$. If $r_1\in M$, then $\rist_G(n)\cap M\ne 1$. If not,by the non-triviality of $r_1$, there must exist $m>n$ such that $r_1\notin \St_G(m)$. In particular, $r_1\notin \rist_G(m)$. Using the fact that $M$ is prodense, there must exist $h\in M$ and $r_2\in \rist_G(m)$ such that $hr_2=r_1$. Therefore, we have $h=r_1r_2^{-1}$. Since $r_1\notin \rist_G(m)$, we know that  $r_1r_2^{-1}\ne 1$, and since $m>n$, we have $\rist_G(m)\leq \rist_G(n)$, and thus $r_1r_2^{-1}\in \rist_G(n)$. We conclude that $r_1r_2^{-1}\in \rist_G(n)\cap M$, and thus this subgroup is non-trivial. This concludes the proof.

\end{proof}

\section{Maximal subgroup of the Basilica group}\label{section:Basilica}

In this section, as an application of Theorem \ref{thm:ProjectionsOFMaximalAreMaximal}, we prove that every maximal subgroup of the Basilica group is of finite index. To the best of our knowledge, this is the first example of a weakly branch but not branch group belonging to the class $\MF$.

\subsection{The Basilica group}

Before we study its maximal subgroups, let us first give a definition of the Basilica group and list some of its properties that will be relevant to our investigation. We refer the interested reader to \cite{GrigorchukZuk02} for more information about this group.

\begin{defn}
Let $X=\{\boldsymbol{0}, \boldsymbol{1}\}$ be an alphabet of two letters, let $\sigma\in \Sym(X)$ be the non-trivial permutation on $X$, and let $a,b\in \Aut(X^*)$ be the automorphisms of the rooted tree $X^*$ defined by the recursive formulas
\[\psi_1(a) = (1,b) \qquad \psi_1(b)=\sigma(a,1).\]
The \emph{Basilica group} $\B=\langle a, b \rangle \leq \Aut(X^*)$ is the group of automorphisms of the rooted tree $X^*$ generated by $a$ and $b$.
\end{defn}

\begin{rem}
In \cite{GrigorchukZuk02}, Grigorchuk and \.{Z}uk defined the Basilica group by a right-action on the rooted tree $X^*$, but it is easy to check that the corresponding left-action, defined by $g\cdot x = x\cdot g^{-1}$, gives rise to the recursion formulas given above.
\end{rem}

The following theorem collects a few properties of the Basilica group $\B$ that were proved by Grigorchuk and \.{Z}uk in \cite{GrigorchukZuk02}. We refer the reader to that text for the proofs.

\begin{thm}\label{thm:FactsAboutBasilica}
Let $\B=\langle a,b \rangle$ be the Basilica group. Then,
\begin{enumerate}[label=(\roman*)]
\item $\B$ is a self-replicating, regular weakly branch group over its derived subgroup $\B'$ (see Definitions \ref{defn:SelfSimilarGroup} and \ref{defn:RegularBranch}),
\item $\B$ is torsion-free,
\item the semigroup generated by $a$ and $b$ is free (so in particular, $\B$ is of exponential growth),
\item $\B/\B' \cong \langle a \rangle \times \langle b \rangle \cong \ZZ^2$.
\end{enumerate}
\end{thm}

In the same article, Grigorchuk and \.{Z}uk also obtained a presentation for $\B$, which will be useful later on.

\begin{thm}[Proposition 9 of \cite{GrigorchukZuk02}]\label{thm:PresentationOfBasilica}
The Basilica group $\B$ has the presentation
\[\B = \langle a,b \mid \lambda^{k}(\tau_m), m=2l+1, k\in \NN, l\in \NN \rangle\]
where
\[\tau_m = [b^{-m}ab^{m}, a]\]
and
\[\lambda \colon \begin{cases}
a\mapsto b^2\\
b\mapsto a.
\end{cases}\]
\end{thm}

\subsection{The derived subgroup of \texorpdfstring{$\B$}{B}}

According to Theorem \ref{thm:FactsAboutBasilica}, the Basilica group $\B$ is a regular weakly branch group over its derived subgroup $\B'$. For this reason, having a good description of $\B'$ will be very useful later on in the study of subgroups and quotients of $\B$. In this subsection, we prove that $\B'$ is a finitely generated group and give a minimal set of generators. We believe that these results could also be of independent interest to anyone studying the Basilica group.

We begin by introducing some notation which will hopefully help to make the computations that follow more readable.

\begin{notation}\label{notation:CommutatorsBasilica}
Let $\B=\langle a, b \rangle$ be the Basilica group, with $a$ and $b$ as above. For all $s,t\in \ZZ$, we will write
\[\alpha_{s,t} = [a^s, b^t],\]
where $[a^s,b^t] = a^{-s}b^{-t}a^sb^t$.
\end{notation}

The following proposition establishes some relations between these commutators.

\begin{prop} \label{prop:CommutatorIdentities}
For all $s,t\in \ZZ$, we have the following relations in $\B$: 
\begin{align*}
\alpha_{s,2t+1} &= \left(\alpha_{1,1}(\alpha_{1,-1}^{-1}\alpha_{1,1})^{t}\right)^{s} \\
\alpha_{s,2t} &= \alpha_{1,1}^{s-1}\left(\alpha_{1,2}^{t}\alpha_{1,1}^{-1}\right)^{s-1}\alpha_{1,2}^t
\end{align*}
\end{prop}
\begin{proof}
The proof is a direct computation, using the injectivity of the map $\psi_1$ defined in Proposition \ref{prop:Psin}. We will make frequent use of the fact that
\begin{align*}
\psi_1(\alpha_{1,1}) &= (a^{-1}ba, b^{-1})\\
\psi_1(\alpha_{1,-1}) &= (b, b^{-1}) \\
\psi_1(\alpha_{1,2}) &= (1, b^{-1}a^{-1}ba) = (1, \alpha_{1,1}^{-1}).
\end{align*}

By direct computation, we find
\[\psi_1(\alpha_{s,2t+1}) = (a^{-t-1}b^sa^{t+1}, b^{-s}).\]
Therefore, it follows from the injectivity of $\psi_1$ that we have
\[\alpha_{s,2t+1} = \alpha_{1,2t+1}^s.\]
Using this fact, we see that
\begin{align*}
\psi_1((\alpha_{1,-1}^{-1}\alpha_{1,1})^t) &= (b^{-1}a^{-1}ba, 1)^t \\
&=(\alpha_{1,1}^{-1},1)^t \\
&=(\alpha_{t,1}^{-1}, 1).
\end{align*}
It follows that
\begin{align*}
\psi_1\left(\left(\alpha_{1,1}(\alpha_{1,-1}^{-1}\alpha_{1,1})^{t}\right)^{s}\right) &= (a^{-1}ba \alpha_{t,1}^{-1}, b^{-1})^s \\
&=(a^{-1}bab^{-1}a^{-t}ba^t, b^{-1})^s \\
&=(\alpha_{1,-1}a^{-t}ba^{t}, b^{-1})^s\\
&=(\alpha_{1,-1}a^{-t}\alpha_{1,-1}^{-1}a^{-1}ba^{t+1}, b^{-1})^s\\
&=((\alpha_{1,-1}a^{-t}\alpha_{1,-1}^{-1}a^t)(a^{-t-1}ba^{t+1}), b^{-1})^{s}\\
&=([\alpha_{1,-1}^{-1}, a^{t}](a^{-t-1}ba^{t+1}), b^{-1})^s.
\end{align*}
Now, using the fact that $\psi_1(\alpha_{1,-1}^{-1}) = (b^{-1}, b)$ and $\psi_1(a^{t}) = (1, b^{t})$, we see that $\alpha_{1,-1}^{-1}$ and $a^{t}$ commute. Therefore, we have
\begin{align*}
\psi_1\left(\left(\alpha_{1,1}(\alpha_{1,-1}^{-1}\alpha_{1,1})^{t}\right)^{s}\right) &=((a^{-t-1}ba^{t+1})^s, b^{-s}) \\
&=\psi_1(\alpha_{s,2t+1}).
\end{align*}
The first relation then immediately follows from the injectivity of $\psi_1$.

To prove the second relation, let us notice that from direct computations, we immediately get
\[\psi_1(\alpha_{s,2t}) = (1, \alpha_{t,s}^{-1}).\]
Hence, we have
\begin{align*}
\psi_1((\alpha_{1,2}^t\alpha_{1,1}^{-1})^{s-1}) &= (a^{-1}b^{-1}a, \alpha_{1,1}^{-t}b)^{s-1}\\
&= (a^{-1}b^{1-s}a, (\alpha_{t,1}^{-1}b)^{s-1}).
\end{align*}
Therefore,
\begin{align*}
\psi_1(\alpha_{1,1}^{s-1}\left(\alpha_{1,2}^{t}\alpha_{1,1}^{-1}\right)^{s-1}\alpha_{1,2}^t) &= (1, b^{1-s}(\alpha_{t,1}^{-1}b)^{s-1}\alpha_{t,1}^{-1}) \\
&=(1, b^{-s+1}(\alpha_{t,1}^{-1}b)^sb^{-1})\\
&=(1, b^{-s+1}([b,a^t]b)^sb^{-1}) \\
&=(1, b^{-s+1}(b^{-1}a^{-t}ba^{t}b)^sb^{-1})\\
&=(1, b^{-s}a^{-t}b^sa^t) \\
&=\psi_1(\alpha_{s,2t}).
\end{align*}
This proves the second relation.
\end{proof}

The previous proposition implies that the derived subgroup of the Basilica group is generated by only three elements.

\begin{prop}\label{prop:GeneratorsOfCommutatorSubgroup}
The derived subgroup $\B'$ of the Basilica group $\B$ is generated by $\alpha_{1,1}=[a,b]$, $\alpha_{1,-1}=[a,b^{-1}]$ and $\alpha_{1,2}=[a,b^2]$.
\end{prop}
\begin{proof}
Since the Basilica group is generated by two elements, $a$ and $b$, its commutator subgroup $\B'$ is generated by the set 
\[\{\alpha_{s,t}=[a^s,b^t] | s,t\in \mathbb{Z} \}.\]
It then follows from Proposition \ref{prop:CommutatorIdentities} that $\B'$ is generated by $\alpha_{1,1}$, $\alpha_{1,-1}$ and $\alpha_{1,2}$.
\end{proof}

It turns out that this generating set is minimal. However, to prove this, we will need a better understanding of the quotients of $\B$, so we postpone the proof of this fact to the next subsection (Proposition \ref{prop:AbelianizationOfB'IsZ3}).

\subsection{Quotients of \texorpdfstring{$\B$}{B}}

To use Theorem \ref{thm:ProjectionsOFMaximalAreMaximal} to study maximal subgroups of the Basilica group $\B$, we first need to show that every proper quotient of $\B$ belongs to the class $\MF$. Let us first notice that $\B$ is not a branch group, so we cannot simply use Proposition \ref{prop:QuotientsOfBranchGroupsAreinMF}. To show this, we first observe that the discrete Heisenberg group $H_3(\ZZ)$ is a quotient of $\B$.

\begin{prop}\label{prop:HeisenbergIsQuotientOfBasilica}
Let $\gamma_3(\B)=[\B',\B]$ be the third term in the lower central series of the Basilica group $\B$, and let $H_3(\ZZ)$ be the discrete Heisenberg group. Then, $\B/\gamma_3(\B) \cong H_3(\ZZ)$.
\end{prop}
\begin{proof}
By Theorem \ref{thm:PresentationOfBasilica} and the fact that $\gamma_3(\B)$ is normally generated by $[[a,b],a]$ and $[[a,b],b]$, we have 
\[\B/\gamma_3(\B) = \langle a,b \mid [[a,b],a], [[a,b],b], \lambda^k(\tau_m), m\in 2\NN+1, k\in \NN \rangle\]
with
\[\tau_m = [b^{-m}ab^{m}, a]\]
and
\[\lambda \colon \begin{cases}
a\mapsto b^2\\
b\mapsto a.
\end{cases}\]
Let $F(a,b)$ be the free group on $a$ and $b$. Then, for all $m,n\in \ZZ$ and for $x,y\in \{a,b\}$, we have
\begin{align*}
[x^{-m}y^{n}x^{m}, y^{n}] &= y^{-n}y^{n}x^{-m}y^{-n}x^{m}y^{-n}x^{-m}y^{n}x^{m}y^{-n}y^{n}y^{n} \\
&=y^{-n}[y^{-n},x^{m}]y^{-n}[x^{m},y^{-n}]y^{n}y^n \\
&=y^{-n}[[x^m,y^{-n}], y^n]y^n\\
&\in \gamma_3(F(a,b)),
\end{align*}
where $\gamma_3(F(a,b))$ is the third term in the lower central series of $F(a,b)$.

Since $\gamma_3(F(a,b))$ is normally generated by $[[a,b],a]$ and $[[a,b],b]$, we conclude by the above that $\lambda^k(\tau_m)$ is a consequence of these two relations for all $k\in \NN$ and $m\in 2\NN+1$. Therefore, we have
\[\B/\gamma_3(\B) = \langle a,b \mid [[a,b],a], [[a,b],b] \rangle\]
which is the presentation of the discrete Heisenberg group $H_3(\ZZ)$.
\end{proof}

As a direct consequence, we get that the Basilica group is not a branch group.

\begin{cor}\label{cor:BasilicaIsNotBranch}
The Basilica group $\B$ is not a branch group.
\end{cor}
\begin{proof}
It follows from Lemma \ref{lemma:SubnormalContainsDerivedRist} that every proper quotient of a branch group is virtually abelian. However, by Proposition \ref{prop:HeisenbergIsQuotientOfBasilica}, $\B$ admits the discrete Heisenberg group $H_3(\ZZ)$ as a quotient, and this quotient must be proper by Lemma \ref{lemma:RistNotSolvable}. As $H_3(\ZZ)$ is not virtually abelian, we conclude that $\B$ is not a branch group.
\end{proof}

Therefore, before we can use Theorem \ref{thm:ProjectionsOFMaximalAreMaximal}, we first need to prove that every proper quotient of $\B$ is in $\MF$. In fact, we will prove something stronger, namely that every proper quotient of $\B$ is virtually nilpotent. This is indeed stronger, since every finitely generated virtually nilpotent group is in $\MF$.

To prove that every proper quotient is virtually nilpotent, since $\B$ is a weakly regular branch group over $\B'$, it is sufficient to prove that $\B/\B''$ is virtually nilpotent, as the next theorem shows.

\begin{thm}\label{thm:QuotientsOfWeaklyBranchGroups}
Let $X$ be a finite alphabet of size $d$, let $G\leq \Aut(X^*)$ be a regular weakly branch group over a normal subgroup $K$, and let $\mathcal{P}$ be a property of groups that is preserved under taking finite direct products, quotients and subgroups. Then, every proper quotient of $G$ is virtually $\mathcal{P}$ (i.e. contains a subgroup of finite index with property $\mathcal{P}$) if and only if $G/K'$ is virtually $\mathcal{P}$.
\end{thm}
\begin{proof}
As $K$ cannot be abelian by Lemmas \ref{lemma:SubnormalContainsDerivedRist} and \ref{lemma:RistNotSolvable}, $G/K'$ is a proper quotient, so the necessity is obvious. Let us show that if $G/K'$ is virtually $\mathcal{P}$, then so is every proper quotient of $G$.

Let $N\trianglelefteq G$ be a non-trivial normal subgroup of $G$. According to Lemma \ref{lemma:SubnormalContainsDerivedRist}, there exists $n\in \NN$ such that $\rist'_G(n) \leq N$. Now, by definition of a regular weakly branch group over $K$, we have that there exists a subgroup $K_n\leq K$ such that $K^{d^n} = \psi_n(K_n)$. In particular, we see that $K_n \leq \rist_G(n)$. Consequently, we have that $K_n' \leq \rist'_G(n) \leq N$.

As $K$ is normal in $G$ and as $G$ is self-similar, it follows from its definition that $K_n$ must also be a normal subgroup of $G$. Consequently, as $K'_n$ is a characteristic subgroup of $K_n$, we have that $K'_n$ is a normal subgroup of $G$. Hence, we can take the quotient $G/K'_n$.

If we can prove that $G/K'_n$ is virtually $\mathcal{P}$, then this will imply that $G/N$ is also virtually $\mathcal{P}$. Indeed, as $K'_n \leq N \leq G$, we have that
\[G/N \cong \left(G/K'_n\right) \bigg/ \left(N/K'_n\right).\]
If $G/K'_n$ is virtually $\mathcal{P}$, then by the correspondence theorem, there exists $H\leq G$ of finite index such that $K'_n\leq H$ and such that  $H/K'_n$ has $\mathcal{P}$. Since $\mathcal{P}$ is preserved by taking quotients, we have that
\[HN/N \cong (H/K'_n)\bigg/(N/K'_n)\cap(H/K'_n)\]
has $\mathcal{P}$. It is also of finite index in $G/N$, since $H$ is of finite index in $G$.

Thus, it suffices to prove that $G/K'_n$ is virtually $\mathcal{P}$. In fact, since $K'_n \leq \St_G(n)$ and since $\St_G(n)$ is of finite index in $G$, it suffices to prove that $\St_G(n)/K'_n$ is virtually $\mathcal{P}$.

Now, since $G$ is self-similar, we have $\psi_n(\St_G(n)) \leq \left(G\right)^{d^n}$. Hence,
\begin{align*}
\psi_n(\St_G(n))/\psi_n(K'_n) &\leq \left(G\right)^{d^n} / (K')^{d^n} \\
&= \left(G/K'\right)^{d^n}. 
\end{align*}

As $G/K'$ is virtually $\mathcal{P}$, there exists a finite index subgroup $H\leq G$ containing $K'$ such that $H/K'$ has property $\mathcal{P}$. Since property $\mathcal{P}$ is preserved by finite direct products, $\left(H/K'\right)^{d^n}$ is a finite index subgroup of $\left(G/K'\right)^{d^n}$ with property $\mathcal{P}$. Let us set $L=\psi_n^{-1}(H^{d^n}\cap \psi_n(\St_G(n)))$. As $K'\leq H$, we clearly have that $K'_n \leq L$. We claim that $L$ is a finite index subgroup of $\St_G(n)$ such that $L/K'_n$ has $\mathcal{P}$.

To see that $L$ is of finite index in $\St_G(n)$, it suffices to notice that since $H^{d^n}$ is of finite index in $G^{d^n}$, we have that $H^{d^n}\cap \psi_n(\St_G(n))$ is of finite index in $\psi_n(\St_G(n))$. Since $\psi_n$ restricted to $\St_G(n)$ is an isomorphism onto its image, we conclude that $L$ is of finite index in $\St_G(n)$. To see that $L/K'_n$ has $\mathcal{P}$, it suffices to notice that $\psi_n$ gives us an isomorphism between $L/K'_n$ and
\[(H^{d^n}\cap \psi_n(\St_G(n)))/(K')^{d^n} \leq (H/K')^{d^n}.\]
Since $(H/K')^{d^n}$ has $\mathcal{P}$ and since $\mathcal{P}$ is inherited by subgroups, we conclude that $L/K'_n$ has $\mathcal{P}$. This proves that $G/K'_n$ is virtually $\mathcal{P}$ and thus concludes the proof.
\end{proof}

We thus only need to show that $\B/\B''$ is virtually nilpotent. For this, we will require the following lemma, which is Lemma 9 in \cite{GrigorchukZuk02}. However, since the proof in \cite{GrigorchukZuk02} relies on Lemma 8 of \cite{GrigorchukZuk02}, whose proof contains a mistake, we give here a different proof.

\begin{lemma}[Lemma 9 of \cite{GrigorchukZuk02}]\label{lemma:B''IsGamma3TimesGamma3}
Let $\gamma_3(\B)=[\B',\B]$ be the third term in the lower central series of the Basilica group $\B$ and let $\B''$ be its second derived subgroup. Then, we have
\[\psi_1(\B'') = \gamma_3(\B)\times\gamma_3(\B).\]
\end{lemma}
\begin{proof}
Let us first show that $\psi_1(\B'')\leq \gamma_3(\B)\times \gamma_3(\B)$. It follows from Proposition \ref{prop:GeneratorsOfCommutatorSubgroup} that $\B''$ is generated by the conjugates in $\B'$ of $[\alpha_{1,1},\alpha_{1,-1}]$, $[\alpha_{1,1},\alpha_{1,2}]$ and $[\alpha_{1,-1},\alpha_{1,2}]$. We find
\begin{align*}
\psi_1([\alpha_{1,1},\alpha_{1,-1}])&=[(a^{-1}ba,b^{-1}), (b,b^{-1})]\\
&=(a^{-1}b^{-1}ab^{-1}a^{-1}bab,1)\\
&=([[b,a],b],1) \in \gamma_3(\B)\times \gamma_3(\B),
\end{align*}
\begin{align*}
\psi_1([\alpha_{1,1},\alpha_{1,2}]) &=[(a^{-1}ba,b^{-1}), (1,[b,a])]\\
&=(1,[b^{-1},[b,a]])\\
&=(1,[[b,a],b^{-1}]^{-1})\in \gamma_3(\B)\times\gamma_3(\B),
\end{align*}
\begin{align*}
\psi_1([\alpha_{1,-1},\alpha_{1,2}])&=[(b,b^{-1}), (1,[b,a])]\\
&=(1,[b^{-1},[b,a]])\\
&=(1,[[b,a],b^{-1}]^{-1})\in \gamma_3(\B)\times\gamma_3(\B).
\end{align*}
Since $\gamma_3(\B)$ is a normal subgroup of $\B$, any conjugate of these elements will also belong to $\gamma_3(\B)\times\gamma_3(\B)$. Therefore, $\B''\leq \gamma_3(\B)\times\gamma_3(\B)$.

Now, let us show that $\gamma_3(\B)\times \gamma_3(\B) \leq \psi_1(\B'')$. Since $\B$ is generated by $a$ and $b$, we have that $\gamma_3(\B)$ is normally generated in $\B$ by $[[b,a],a]$ and $[[b,a],b]$. Since
\[[[b,a],a] = [(a^{-1}b^{-1}a,b), (1,b)] = 1,\]
we conclude that $\gamma_3(\B)$ is normally generated by $[[b,a],b]$. We have computed above that 
\[\psi_1([\alpha_{1,1},\alpha_{1,-1}])=([[b,a],b],1) \in \gamma_3(\B)\times 1.\]
Since $\B''$ is normal in $\B$ and since $\B$ is self-replicating, we conclude that $\gamma_3(\B)\times 1 \leq \psi_1(\B'')$. Conjugating by $b$, we then get that $1 \times \gamma_3(\B) \leq \psi_1(\B'')$, from which we conclude that $\gamma_3(\B)\times \gamma_3(\B) \leq \psi_1(\B'')$.
\end{proof}

Using this lemma, we can show that $\B/\B''$ is virtually nilpotent.

\begin{lemma}\label{lemma:BOverB''VirtuallyNilpotent}
The group $\B/\B''$ is virtually nilpotent.
\end{lemma}
\begin{proof}
Thanks to Lemma \ref{lemma:B''IsGamma3TimesGamma3}, we have
\[\psi_1(\St_{\B}(1))/\psi_1(\B'') \leq \left(\B/\gamma_3(\B)\right)\times \left(\B/\gamma_3(\B)\right).\]
As the group $\left(\B/\gamma_3(\B)\right)\times \left(\B/\gamma_3(\B)\right)$ is nilpotent, $\psi_1(\St_{\B}(1))/\psi_1(\B'')$ is nilpotent. Since $\psi_1$ is injective, we have
\[\psi_1(\St_{\B}(1))/\psi_1(\B'')\cong \St_{\B}(1)/\B''.\]
As $\St_{\B}(1)$ is of finite index in $\B$, $\St_{\B}(1)/\B''$ is of finite index in $\B/\B''$. Hence, we found a nilpotent subgroup of finite index in $\B/\B''$.
\end{proof}

We can now finally conclude that every proper quotient of the Basilica group is virtually nilpotent, and thus in $\MF$.

\begin{prop}\label{prop:BasilicaIsJustNonVirtuallyNilpotent}
Every proper quotient of the Basilica group $\B$ is virtually nilpotent. In particular, every proper quotient of $\B$ is in $\MF$.
\end{prop}
\begin{proof}
The fact that every proper quotient of Basilica is virtually nilpotent follows directly from Theorem \ref{thm:QuotientsOfWeaklyBranchGroups}, Lemma \ref{lemma:BOverB''VirtuallyNilpotent} and the fact that $\B$ is a regular weakly branch group over $\B'$. Since every finitely generated virtually nilpotent group is in $\MF$, we conclude that every proper quotient of $\B$ is in $\MF$.
\end{proof}

We can thus use Theorem \ref{thm:ProjectionsOFMaximalAreMaximal} to study maximal subgroups of the Basilica group, which we will do in the next section. Before we go on, however, let us quickly remark that by combining Lemma \ref{lemma:B''IsGamma3TimesGamma3} and Proposition \ref{prop:HeisenbergIsQuotientOfBasilica}, we can show that the set of generators of $\B'$ obtained in Proposition \ref{prop:GeneratorsOfCommutatorSubgroup} is minimal.

\begin{prop}\label{prop:AbelianizationOfB'IsZ3}
Let $\B$ be the Basilica group, $\B'$ be its derived subgroup and $\B''$ be its second derived subgroup. The map $\ZZ^3 \rightarrow \B'/\B''$ defined by sending the canonical generators of $\ZZ^3$ to $\alpha_{1,1},\alpha_{1,-1}$ and $\alpha_{1,2}$ is an isomorphism (where $\alpha_{1,1},\alpha_{1,-1}$ and $\alpha_{1,2}$ are defined in Notation \ref{notation:CommutatorsBasilica}). In particular, $\alpha_{1,1}$, $\alpha_{1,-1}$ and $\alpha_{1,2}$ form a minimal set of generators for $\B'$.
\end{prop}
\begin{proof}
Since $\psi_1$ is an injective map, we have
\[\B'/\B'' \cong \psi_1(\B')/\psi_1(\B'').\]
Now, by Lemma \ref{lemma:B''IsGamma3TimesGamma3}, we have $\psi_1(\B'') = \gamma_3(\B)\times \gamma_3(\B)$, where $\gamma_3(\B)=[\B',\B]$. It thus follows from Proposition \ref{prop:HeisenbergIsQuotientOfBasilica} that $\psi_1(\B')/\psi_1(\B'') \leq H_3(\ZZ) \times H_3(\ZZ)$, where $H_3(\ZZ)$ is the discrete Heisenberg group.

Let $f\colon \ZZ^3 \rightarrow \B'/\B''$ be the homomorphism sending $(1,0,0)$ to $\alpha_{1,1}$, $(0,1,0)$ to $\alpha_{1,-1}$ and $(0,0,1)$ to $\alpha_{1,2}$, and let $g\colon \B'/\B'' \rightarrow H_3(\ZZ) \times H_3(\ZZ)$ be the injective homomorphism implied above. To prove the result, it suffices to show that the kernel of $g\circ f$ is trivial.

By direct computation, we see that
\[g(\alpha_{1,1}) = (bc^{-1}, b^{-1}), \quad g(\alpha_{1,-1}) = (b,b^{-1}), \quad g(\alpha_{1,2})=(1, c^{-1})\]
where $H_3(\ZZ) = \langle a,b \mid [[a,b],a], [[a,b],b] \rangle$ and $c=[a,b]$. Now, let $(l,m,n)\in \ker(g\circ f)$ be an arbitrary element of the kernel of $g\circ f$. It follows from the above computations that
\[g\circ f (l,m,n) = (b^{l+m}c^{-l}, b^{-l-m}c^{-n})\]
and we quickly see that this is trivial if and only if $l=m=n=0$. Thus, $g\circ f$ is injective, which implies that $f$ is injective. By Proposition \ref{prop:GeneratorsOfCommutatorSubgroup}, it is also surjective and is thus an isomorphism.
\end{proof}

\subsection{Maximal subgroups of the Basilica group}

In this subsection, we use Theorem \ref{thm:ProjectionsOFMaximalAreMaximal} to prove that every maximal subgroup of the Basilica group $\B$ is of finite index. Although the details are very different, the general strategy of the proof is roughly the same as the one used by Pervova in \cite{Pervova05} to prove that the Grigorchuk group is in $\MF$. Namely, assuming that there exists a maximal subgroup of infinite index $M<\B$, we will show, using arguments of length reduction, that there must exist some vertex $v\in X^*$ such that $M_v=G$, thus contradicting Theorem \ref{thm:ProjectionsOFMaximalAreMaximal}.

To achieve this, we will require several intermediate steps. Before we begin, however, let us first fix some notation that will be useful throughout this section.

\begin{notation}
We will denote by $|\cdot|\colon \B \rightarrow \NN$ the \emph{word norm} with respect to the generating set $S=\{a,a^{-1},b,b^{-1}\}$. In other words, for $g\in \B$, we denote by $|g|$ the smallest word in the alphabet $S$ representing $g$. In what follows, we will generally make no distinction in the notation between a word in the generating set $S$ and the element it represents in the group $\B$ and rely on the context to distinguish between the two cases. In particular, if $w\in S^*$ is a word in the alphabet $S$, we will denote by $|w|$ the length of the corresponding element in $\B$, which may be smaller that the number of letters in $w$. A word $w=s_1\dots s_n\in S^*$ will be called a \emph{word of minimal length} or a \emph{geodesic word} if $|s_1\dots s_n| = n$.
\end{notation}

\begin{rem}
The map $|\cdot|\colon \B \rightarrow \NN$ is subadditive: for $g_1,g_2\in \B$, we have $|g_1g_2|\leq |g_1|+|g_2|$. In particular, if $w=s_1\dots s_is_{i+1}\dots s_n$ is a geodesic word, this implies that $|w|=|s_1\dots s_i|+|s_{i+1}\dots s_n|$ for all $1\leq i \leq n$. This means that subwords of geodesic words are always geodesic.
\end{rem}

\begin{notation}
Let $g,g_1,g_2\in \B$ and $\epsilon \in \{0,1\}$ be such that $\psi_1(g)=\sigma^{\epsilon}(g_1,g_2)$. From now on, we will frequently suppress the $\psi_1$ from the notation and simply write $g=\sigma^\epsilon (g_1,g_2)$.
\end{notation}

We begin our investigation with a series of lemmas establishing various bounds between the length of elements in $\B$ and the length of their projections.

\begin{lemma}\label{lemma:BasilicaNonEll1Expanding}
Let $g\in \B$ be an arbitrary element of the Basilica group, and let $g_1,g_2\in \B$, $\epsilon \in \{0,1\}$ be such that $g=\sigma^{\epsilon}(g_1,g_2)$. Then, $|g_1| + |g_2| \leq |g|$.
\end{lemma}
\begin{proof}
As $a=(1,b)$, $b=\sigma(a,1)$, $a^{-1} = (1,b^{-1})$ and $b^{-1} = \sigma(1,a^{-1})$, we see that the given inequality is true for the generating set $S=\{a,b,a^{-1},b^{-1}\}$. Therefore, by induction, it must be true for any element of $\B$.
\end{proof}

\begin{lemma}\label{lemma:LengthOfChildrenOfSquare}
Let $g=\sigma(g_1, g_2) \notin \St_\B(1)$ be an element of $\B$ that does not belong to the stabiliser of the first level, and let $\alpha, \beta\in \B$ be such that $g^2 = (\alpha, \beta)$. Then, $|\alpha|, |\beta| \leq |g|$.
\end{lemma}
\begin{proof}
We have $g^2 = \sigma(g_1,g_2)\sigma(g_1,g_2) = (g_2g_1, g_1g_2)$. Hence, thanks to Lemma \ref{lemma:BasilicaNonEll1Expanding}, we have $|\alpha| \leq |g_2| + |g_1| \leq |g|$, and likewise, $|\beta|\leq |g|$.
\end{proof}

\begin{lemma}\label{lemma:ForbiddenSubwordBBinverse}
Let $g=\sigma^{\epsilon}(g_1,g_2) \in \B$ be an arbitrary element of the Basilica group, where $\epsilon\in \{0,1\}$, and let $x_1x_2\dots x_n \in S^*$ be a word of minimal length representing $g$, where $S=\{a,b,a^{-1},b^{-1}\}$. If there exist $1\leq i<j\leq n$ such that $x_i=b$, $x_j=b^{-1}$, then $|g_1| + |g_2| < |g|=n$.
\end{lemma}
\begin{proof}
As the word $x_1x_2\dots x_n$ is reduced (otherwise, it would not be of minimal length), it follows from the hypothesis that it must contain a subword of the form $ba^kb^{-1}$ for some $k\in \ZZ^*$. Seen as an element of $\B$, we have
\[ba^{k}b^{-1} = \sigma(a,1)(1,b^{k})(a^{-1},1)\sigma = (b^k,1).\]
Since $ba^kb^{-1}$ is a subword of a geodesic word, we must have $|ba^kb^{-1}| = |k|+2$. Indeed, otherwise, we could replace it by a shorter word representing the same element. On the other hand, $|b^k| \leq |k|$. Thus, there is a difference of at least $2$ between the length of $ba^{k}b^{-1}$ and the sum of the length of its children. By using subadditivity, Lemma \ref{lemma:BasilicaNonEll1Expanding}, and the fact that every subword of a geodesic word must again be a geodesic, we can conclude that $|g_1|+|g_2| \leq |g|-2 < |g|$.
\end{proof}

\begin{lemma}\label{lemma:ForbiddenWordBsquared}
Let $g=\sigma^{\epsilon}(g_1,g_2) \in \B$ be an arbitrary element of the Basilica group, where $\epsilon\in \{0,1\}$, and let $x_1x_2\dots x_n\in S^*$ be a word in the alphabet $S=\{a,b,a^{-1},b^{-1}\}$ of minimal length representing $g$. If $x_1x_2\dots x_n$ contains a subword of the form $b^{-2}a^{k}b^2$, then $|g_1| + |g_2| < |g| = n$.
\end{lemma}
\begin{proof}
In $\B$, we have
\[b^{-2}a^kb^2 = (a^{-1},a^{-1})(1,b^{k})(a,a) = (1,a^{-1}b^{k}a).\]
As in the proof of Lemma \ref{lemma:ForbiddenSubwordBBinverse}, we observe that $|b^{-2}a^{k}b^{2}| = |k|+4$ and $|1|+|a^{-1}b^ka|\leq k+2$ and thus conclude that $|g_1| + |g_2| \leq n-2 < n=|g|$.
\end{proof}

In addition to these facts regarding length contraction of elements of $\B$, we will also need to know the equivalence classes of the projections of some elements modulo the commutator subgroup $\B'$, which we study in the following lemma.

\begin{notation}
Let $g_1,g_2\in \B$ be two arbitrary elements. We will write $g_1\equiv_{\B'} g_2$ if $g_1\B' = g_2\B'$.
\end{notation}

\begin{lemma}\label{lemma:CongruenceOfProjections}
Let $g\notin \St_\B(1)$ and $g^2=(g_1,g_2)$. Then,
\begin{align*}
g\equiv_{\B'} ab &\Rightarrow g_1\equiv_{\B'} g_2 \equiv_{\B'} ab\\
g\equiv_{\B'} ab^{-1} &\Rightarrow g_1\equiv_{\B'} g_2 \equiv_{\B'} a^{-1}b \\
g\equiv_{\B'} a^{-1}b &\Rightarrow g_1\equiv_{\B'} g_2 \equiv_{\B'} ab^{-1}.
\end{align*}
\end{lemma}
\begin{proof}
If $g=abz$ for some $z\in \B'$ with $z=(z_1,z_2)$, then
\[g^2 = (1,b)\sigma(a,1)(z_1,z_2)(1,b)\sigma(a,1)(z_1,z_2) = (z_2baz_1,baz_1z_2).\]
According to Lemma 5 of \cite{GrigorchukZuk02}, we have $z_1 \equiv_{\B'} z_2^{-1}$, so the result follows.
Similarly, if $g=ab^{-1}z$, we have
\[g^2 = (1,b)(a^{-1},1)\sigma(z_1,z_2)(1,b)(a^{-1},1)\sigma(z_1,z_2) = (a^{-1}z_2bz_1, bz_1a^{-1}z_2),\]
and if $g=a^{-1}bz$, we have
\[g^2=(1,b^{-1})\sigma(a,1)(z_1,z_2)(1,b^{-1})\sigma(a,1)(z_1,z_2) = (z_2b^{-1}az_1, b^{-1}az_1z_2).\]
\end{proof}

We are now almost in position to prove that any subgroup of $\B$ that is prodense must project to $\B$ on some vertex. The next few lemmas form the crucial part of the proof.

\begin{lemma}\label{lemma:CanFindAB}
Let $g\in \B$ be such that $g\equiv_{\B'} ab $. Then, there exist a vertex $u\in X^*$ in the rooted tree $X^*$ and an element $g' \in \St_\B(u) \cap \langle g \rangle$ such that $\varphi_u(g') = ab$.
\end{lemma}
\begin{proof}
Let us proceed by induction on the length of $g$.

By definition, the elements of length $1$ of $\B$ are $a,b,a^{-1},b^{-1}$, none of which are congruent to $ab$ modulo $\B'$ by Theorem \ref{thm:FactsAboutBasilica}, so the case $|g|=1$ is impossible. For $|g|=2$, by the same theorem, the only possibilities are $g=ab$ or $g=ba$. The case $g=ab$ is trivial. If $g=ba$, we have $g^2 = baba = (ba,ab)$, and so $\varphi_{\1}(g^2)=ab$.

Now, let us assume that the result is true for any $h\in \B$ such that $h\equiv_{\B'} ab$ and $|h|<n$ for some $n\in \NN$, and let $g\in \B$ be such that $g\equiv_{\B'} ab$ and $|g|=n$.

Since $g\equiv_{\B'} ab$, we must have $g\notin \St_\B(1)$, so $g=\sigma(g_1,g_2)$. Therefore, we have $g^2 = (g_2g_1, g_1g_2)$. By Lemma \ref{lemma:CongruenceOfProjections}, $g_2g_1\equiv_{\B'} g_1g_2 \equiv_{\B'} ab$, and by Lemma \ref{lemma:LengthOfChildrenOfSquare}, $|g_2g_1|, |g_1g_2| \leq |g|=n$. If $|g_2g_1|<n$ or $|g_1g_2|<n$, we can then conclude by induction. Otherwise, we must have $|g_2g_1|=|g_1g_2|=n$. Therefore, the words representing $g_1$ and $g_2$ obtained from a geodesic of $g$ by the substitution $a\mapsto (1,b)$ and $b\mapsto \sigma(a,1)$ must be geodesics, and so must their concatenations $g_2g_1$ and $g_1g_2$ (since the sum of the length of the words for $g_1$ and $g_2$, before any reduction, is exactly $n$).

Let us write $g_1g_2 = \sigma(\alpha,\beta)$. If the geodesic word for $g_1$ discussed above contains $b$ and the one for $g_2$ contains $b^{-1}$, then by Lemma \ref{lemma:ForbiddenSubwordBBinverse}, $|\alpha| + |\beta|<n$. Therefore, $(g_1g_2)^2 = (\beta\alpha,\alpha\beta)$ with $|\alpha\beta|<n$, $\alpha\beta \equiv_{\B'} ab$. Hence, we can conclude by induction. Likewise, if $g_2$ contains $b$ and $g_1$ contains $b^{-1}$, we can conclude by induction by using the projections of $(g_2g_1)^2$.

Since it follows from Theorem \ref{thm:FactsAboutBasilica} that the sum of the exponents of $a$ in any word representing $g$ is $1$, the exponents of $b$ in $g_1$ and $g_2$ must sum up to $1$. Hence, if $g_1$ and $g_2$ both contain some $b$, one of them must also contain $b^{-1}$. Likewise, if both contain some $b^{-1}$, then one of them must contain $b$. Hence, the only remaining case is if $g_1=a^k$ or $g_2=a^k$ for some $k\in \ZZ$, with $|g_1g_2|=|g_2g_1|=|g|$. We will show that this can only occur if $g=ab$ or $g=ba$.

Let us notice that $a^{k_1}b^{2l+1}a^{k_2} = \sigma(b^{k_1}a^{l+1}, a^lb^{k_2})$. Hence, if $g$ contains a subword of the form $a^{k_1}b^{2l+1}a^{k_2}$ with $k_1,k_2\in \ZZ^*$ and $l\in \ZZ$, then both $g_1$ and $g_2$ contain some non-trivial power of $b$. Hence, if $g_1=a^{k}$ or $g_2=a^{k}$, then we must have
\[g=b^{2l_1+1}a^{k_1}b^{2l_2}a^{k_2} \dots b^{2l_i}a^{k_i}\]
or
\[g=a^{k_1}b^{2l_1}a^{k_2}b^{2l_2} \dots a^{k_i}b^{2l_i+1}\]
with $\sum_{j=1}^{i} l_j = 0$ and $\sum_{j=1}^{i}k_j = 1$. Indeed, we just saw that in a geodesic word representing $g$, odd powers of $b$ cannot be sandwiched between non-zero powers of $a$. This means that odd powers of $b$ must be either at the very beginning or at the very end of the word. Hence, there are only two possible positions, which implies that there are at most two odd powers of $b$. As the sum of the powers of $b$ must be $1$, we conclude that the word for $g$ must contain exactly one $b$ with an odd power, either at the beginning or at the end, thus obtaining the two possibilities above.

If $g=b^{2l_1+1}a^{k_1}b^{2l_2}a^{k_2} \dots b^{l_i}a^{k_i}$, it follows from Lemmas \ref{lemma:ForbiddenSubwordBBinverse} and \ref{lemma:ForbiddenWordBsquared} that $g=ba$ or $g=b^{-1}a^{k_1}b^2a^{k_2}$ with $k_1+k_2=1$. Indeed, otherwise, $g$ would contain a subword of the form $ba^{k}b^{-1}$ or $b^{-2}a^{k}b^2$, which contradicts the hypothesis that $|g_1g_2| = |g_2g_1| = n$. If $g=b^{-1}a^{k_1}b^2a^{k_2}$, we have $g^2 = (a^{-1}b^{k_1}ab^{k_2}a, b^{k_1}ab^{k_2})$, and $|b^{k_1}ab^{k_2}| \leq |k_1|+|k_2| + 1 < |k_1|+|k_2|+3 = |g|$, a contradiction. Hence, the only possible case is $g=ba$.

Similarly, if $g=a^{k_1}b^{l_1}a^{k_2}b^{l_2} \dots a^{k_i}b^{2l_i+1}$, then unless $g=ab$, $g$ must contain a subword of the form $ba^kb^{-1}$ or $b^{-2}a^{k}b^2$, which is impossible according to Lemmas \ref{lemma:ForbiddenSubwordBBinverse} and \ref{lemma:ForbiddenWordBsquared}.

This concludes the proof.
\end{proof}

\begin{lemma}\label{lemma:CanFindBInverseA}
Let $g\in \B$ be such that $g\equiv_{\B'} ab^{-1}$. Then, there exist a vertex $u\in X^*$ and an element $g' \in \St_\B(u) \cap \langle g \rangle$ such that $\varphi_u(g') = b^{-1}a$.
\end{lemma}
\begin{proof}
We again proceed by induction on $|g|$.

The case $|g|=1$ is impossible. If $|g|=2$, we have $g=b^{-1}a$ or $g=ab^{-1}$. Since $(ab^{-1})^{-2} = \sigma(a,1)(1,b^{-1})\sigma(a,1)\sigma(1,b^{-1}) = (b^{-1}a,ab^{-1})$, the result is true in those cases.

Let us now assume that the result is true for elements of length smaller than $n\in\NN$ and let $g\in \B$ be such that $g\equiv_{\B'} ab^{-1}$ and $|g|=n$. Writing $g=\sigma(g_1,g_2)$, $g_1g_2=(\alpha,\beta)$ and $g_2g_1 = (\alpha',\beta')$, if $|\alpha|, |\beta|, |\alpha'|$ or $|\beta'|$ is smaller than $n$, we find that the result is true by induction thanks to Lemma \ref{lemma:CongruenceOfProjections} and Lemma \ref{lemma:LengthOfChildrenOfSquare}.

Notice that once again, unless $g_1=a^k$ or $g_2=a^k$ for some $k\in \ZZ$, then one of $|\alpha|, |\beta|, |\alpha'|$ or $|\beta'|$ must be smaller than $n$, thanks to Lemma \ref{lemma:ForbiddenSubwordBBinverse} and the fact that the exponents of $b$ in $g_1$ and $g_2$ must sum to $1$.

As in the proof of Lemma \ref{lemma:CanFindAB}, this means that $g$ cannot contain a subword of the form $a^{k_1}b^{2l+1}a^{k_2}$ with $k_1,k_2\in\ZZ^*$ and $l\in \ZZ$. Therefore, we must have 
\[g=b^{2l_1-1}a^{k_1}b^{2l_2}a^{k_2} \dots b^{2l_i}a^{k_i}\]
or
\[g=a^{k_1}b^{2l_1}a^{k_2}b^{l_2} \dots a^{k_i}b^{2l_i-1}\]
with $\sum_{j=1}^{i} l_j = 0$ and $\sum_{j=1}^{i}k_j = 1$.

If $g=b^{2l_1-1}a^{k_1}b^{2l_2}a^{k_2} \dots b^{2l_i}a^{k_i}$, then unless $g=b^{-1}a$, $g$ must contain a subword of the form $ba^{k}b^{-1}$ or $b^{-2}a^kb^2$, which is impossible according to Lemmas \ref{lemma:ForbiddenSubwordBBinverse} and \ref{lemma:ForbiddenWordBsquared}.

If $g=a^{k_1}b^{2l_1}a^{k_2}b^{2l_2} \dots a^{k_i}b^{2l_i-1}$, then for the same reasons, we must have $g=ab^{-1}$ or $g=a^{k_1}b^{-2}a^{k_2}b$ with $k_1+k_2=1$. However, $(a^{k_1}b^{-2}a^{k_2}b)^2 = (a^{-1}b^{k_1}a^{-1}b^{k_2}a,b^{k_1}a^{-1}b^{k_2})$, and $|b^{k_1}a^{-1}b^{k_2}|\leq |k_1|+|k_2| + 1 < |g|$.

Hence, unless $g=ab^{-1}$ or $g=b^{-1}a$, we always have that one of $\alpha, \beta, \alpha', \beta'$ is of length smaller than $|g|$. We can therefore conclude by induction thanks to Lemma \ref{lemma:CongruenceOfProjections}.
\end{proof}

\begin{lemma}\label{lemma:ABPersists}
Let $u$ be a vertex of the rooted tree $X^*$. Then, there exists $g\in \langle ab \rangle \cap \St_\B(u)$ such that $\varphi_u(g) = ab$ or $\varphi_u(g) = ba$.
\end{lemma}
\begin{proof}
We have $(ab)^2 = (ba,ba)$ and $(ba)^2 = (ba,ab)$. The result follows by induction.
\end{proof}

We can now finally piece all of these lemmas together to show that prodense subgroups of the Basilica group must project to $\B$ on some vertex.

\begin{prop}\label{prop:ProjectionsOfDenseAreNotProperForBasilica}
Let $H\leq \B$ be a subgroup such that $HN=\B$ for all non-trivial normal subgroups $N\trianglelefteq \B$ (in other words, $H$ is prodense). Then, there exists a vertex $u\in X^*$ such that $H_u = \B$, where, as in Notation \ref{notation:Projections}, $H_u=\varphi_u(\St_H(u))$.
\end{prop}
\begin{proof}
Since $H\B'=\B$, there exists $g\in H$ such that $g\equiv_{\B'} ab$. Hence, it follows from Lemma \ref{lemma:CanFindAB} that there exists $v\in X^*$ such that $ab\in H_v$. Now, by Lemma \ref{lemma:ProjectionsOfDenseAreDense}, $H_v$ is prodense in $\B$, so we have $H_v\B'=\B$. Hence, there exists $h\in H_v$ such that $h\equiv_{\B'} ab^{-1}$. Therefore, according to Lemma \ref{lemma:CanFindBInverseA}, there exists $v'$ such that $b^{-1}a\in (H_v)_{v'} = H_{vv'}$. From Lemma \ref{lemma:ABPersists}, we also have that either $ab\in H_{vv'}$ or $ba\in H_{vv'}$.

If $ab, b^{-1}a \in H_{vv'}$, then $a^2 \in H_{vv'}$. Since $a^2=(1,b^2)$ and $b^2 = (a,a)$, if we set $u=vv'\1\1 \in X^*$, we have that $a$ and either $ab$ or $ba$ are in $H_u$. Since $\B$ is generated by $a$ and $b$, we get $H_u=\B$.

Likewise, if $ba,b^{-1}a \in H_{vv'}$, then $b^2\in H_{vv'}$, and since $b^2=(a,a)$, by setting $u=vv'\1\in X^*$, we get that $a,b \in H_u$, so $H_u=\B$.
\end{proof}

We are now finally ready to prove the announced theorem.

\begin{thm}\label{thm:BasilicaIsInMF}
Every maximal subgroup of the Basilica group $\B$ is of finite index.
\end{thm}
\begin{proof}
Suppose that there exists a maximal subgroup $M<\B$ of infinite index. By Proposition \ref{prop:BasilicaIsJustNonVirtuallyNilpotent}, every proper quotient of $\B$ is in $\mathcal{MF}$. Therefore, $M$ cannot contain any non-trivial normal subgroup of $\B$, so we must have $MN=\B$ for every non-trivial normal subgroup $N\trianglelefteq \B$. This implies, thanks to Proposition \ref{prop:ProjectionsOfDenseAreNotProperForBasilica}, that there exists $u\in X^*$ such that $M_u=\B$. However, by Theorem \ref{thm:ProjectionsOFMaximalAreMaximal}, $M_u$ must be a maximal subgroup of infinite index of $\B$, which is a contradiction. Hence, $\B$ admits no maximal subgroup of infinite index.
\end{proof}

As was mentioned above, in addition to being one of the rare non-linear examples of a group belonging to the class $\MF$, the Basilica group is, to the best of our knowledge, the first example of a weakly branch but not branch group in this class. It is also different in many other aspects from the groups studied by Pervova \cite{Pervova05} and their generalisations \cite{AlexoudasKlopschThillaisundaram16, KlopschThillaisundaram18}, such as the Grigorchuk group and the GGS groups. To name but a few, it is torsion-free, it is not just-infinite and it admits non-nilpotent quotients. This example thus serves to illustrate the wide range of algebraic properties that can be enjoyed by weakly branch groups in $\MF$.

\bibliography{biblio}
\bibliographystyle{plain}

\end{document}